\documentclass[final,leqno]{siamltex704}
\usepackage{mathrsfs}
\usepackage{amsfonts,amssymb}
\usepackage{dsfont}
\usepackage{pifont}
\usepackage{hyperref}
\usepackage{multirow}
\usepackage{amsmath}
\usepackage{amssymb}
\usepackage{graphicx}
\usepackage{float}
\usepackage[notcite,notref]{showkeys}

\newtheorem{algorithm}{Weak Galerkin Algorithm}

\newcommand{\bu}{{\bf u}}

\newcommand{\bw}{{\bf w}}

\newcommand{\be}{{\bf e}}
\newcommand{\bv}{{\bf v}}

\def\T{{\mathcal T}}
\def\E{{\mathcal E}}

\def\W{{\mathcal W}}
\def\pT{{\partial T}}
\def\l{{\langle}}
\def\r{{\rangle}}

\def\dw{{\mathcal \nabla_w\cdot}}

\def\T{{\mathcal T}}
\def\E{{\mathcal E}}

\def\A{{\kappa^{-1}}}
\def\jump#1{{[\![#1[\!]}}
\def\bbf{{\bf f}}

\def\bn{{\bf n}}

\def\bbQ{\mathbb{Q}}

\def\3bar{{|\hspace{-.02in}|\hspace{-.02in}|}}

\title{A stable numerical algorithm for the Brinkman equations
by weak Galerkin finite element methods}
\author{Lin Mu\thanks{Department of Mathematics, Michigan State University,
East Lansing, MI 48824 (linmu@msu.edu)}\and Junping
Wang\thanks{Division of Mathematical Sciences, National Science
Foundation, Arlington, VA 22230 (jwang@\break nsf.gov). The research
of Wang was supported by the NSF IR/D program, while working at the
Foundation. However, any opinion, finding, and conclusions or
recommendations expressed in this material are those of the author
and do not necessarily reflect the views of the National Science
Foundation.} \and Xiu Ye\thanks{Department of Mathematics,
University of Arkansas at Little Rock, Little Rock, AR 72204
(xxye@ualr.edu). This research was supported in part by National
Science Foundation Grant DMS-1115097}}
\begin{document}
\maketitle

\begin{abstract}
This paper presents a stable numerical algorithm for the Brinkman
equations by using weak Galerkin (WG) finite element methods. The
Brinkman equations can be viewed mathematically as a combination of
the Stokes and Darcy equations which model fluid flow in a
multi-physics environment, such as flow in complex porous media with
a permeability coefficient highly varying in the simulation domain.
In such applications, the flow is dominated by Darcy in some regions
and by Stokes in others. It is well known that the usual Stokes
stable elements do not work well for Darcy flow and vise versa. The
challenge of this study is on the design of numerical schemes which
are stable for both the Stokes and the Darcy equations. This paper
shows that the WG finite element method is capable of meeting this
challenge by providing a numerical scheme that is stable and
accurate for both Darcy and the Stokes dominated flows. Error
estimates of optimal order are established for the corresponding WG
finite element solutions. The paper also presents some numerical
experiments that demonstrate the robustness, reliability,
flexibility and accuracy of the WG method for the Brinkman
equations.
\end{abstract}


\pagestyle{myheadings}

\section{Introduction} This paper is concerned with the development of
stable numerical methods for the Brinkman equations by using weak
Galerkin finite element methods. The Brinkman equations model fluid
flow in complex porous media with a permeability coefficient highly
varying so that the flow is dominated by Darcy in some regions and
by Stokes in others. In a simple form, the Brinkman model seeks
unknown functions $u$ and $p$ satisfying
\begin{eqnarray}
-\mu\Delta{\bf u}+\nabla p+\mu\kappa^{-1}{\bf u}&=& \bbf\quad
\mbox{in}\;\Omega,\label{moment}\\
\nabla\cdot\bu &=&0\quad \mbox{in}\;\Omega,\label{cont}\\
{\bf u}&=&{\bf g}, \mbox{ on }\; \partial\Omega,\label{bc}
\end{eqnarray}
where $\mu$ is the fluid viscosity and $\kappa$ denotes the
permeability tensor of the porous media which occupies a polygonal
or polyhedral domain $\Omega$ in $\mathbb{R}^d\; (d=2,3)$. $u$ and
$p$ represent the velocity and the pressure of the fluid, and $f$ is
a momentum source term. For simplicity, we consider (\ref{moment})
and (\ref{cont}) with ${\bf g}=0$ and $\mu=1$ (note that one can
always scale the solution with $\mu$).

Assume that there exist two positive numbers $\lambda_1,
\lambda_2>0$ such that
\begin{equation}\label{ellipticity}
\lambda_1 \xi^t\xi\le \xi^t \kappa^{-1} \xi\le \lambda_2
\xi^t\xi,\qquad\forall \xi\in\mathbb{R}^d.
\end{equation}
Here $\xi$ is understood as a column vector and $\xi^t$ is the
transpose of $\xi$. We consider the case where $\lambda_1$ is of
unit size and $\lambda_2$ is possibly of large size.

The Brinkman equations (\ref{moment}) and (\ref{cont}) are used to
model fluid motion in porous media with fractures. The model can
also be regarded as a generalization of the Stokes equations that
represent a valid approximation of the Navier-Stokes equations at
low Reynolds numbers. Modeling fluid flow in complex media with
multiphysics has significant impact for many industrial and
environmental problems such as industrial filters, open foams, or
natural vuggy reservoirs. The permeability with high contrast
determines that flow velocity may vary greatly through porous media.
Mathematically, the Brinkman equations can be viewed as a
combination of the Stokes and the Darcy equations, but with change
of type from place to place in the computational domain. Due to the
type change, numerical schemes for the Brinkman equations must be
carefully designed to accommodate both the Stokes and Darcy
simultaneously. The numerical experiments in \cite{mtw} indicate
that the convergent rate deteriorates as the Brinkman becomes
Darcy-dominating when certain stable Stokes elements are used; such
elements include the conforming $P_2$-$P_0$ element, the
nonconforming Crouzeix-Raviart element, and the Mini element.
Similarly, the convergent rate deteriorates as the Brinkman is
Stokes-dominating when Darcy stable elements such as the lowest
order Raviart–Thomas element \cite{mtw} are used.

The main challenge for solving Brinkman equations is in the
construction of numerical schemes that are stable for both the Darcy
and the Stokes equations. In literature, a great deal of effort has
been made in meeting this challenge by modifying either existing
Stokes elements or Darcy elements to obtain new Brinkman stable
elements. For example, methods based on Stokes elements have been
studied in \cite{bc} and methods based on Darcy elements can be
found in \cite{ks,mtw}.

Weak Galerkin (WG) is a general finite element technique for partial
differential equations in which differential operators are
approximated by their weak forms as distributions. WG methods, by
design, make use of discontinuous piecewise polynomials on finite
element partitions with arbitrary shape of polygons and polyhedrons.
The flexibility of WG on the selection of approximating polynomials
makes it an excellent candidate for providing stable numerical
schemes for PDEs with multi-physics properties. The weak Galerkin
method was first introduced in \cite{wy,wy-mixed} for the second
order elliptic problem.

The goal of this paper is to develop a stable weak Galerkin finite
element method for the Brinkman equations. In Section
\ref{Section:section2}, a WG finite element scheme will be
introduced for the Brinkman model. It demonstrates that WG offers a
natural and straightforward framework for constructing stable
numerical algorithms for the Brinkman equations. In Section
\ref{Section:errores}, an optimal order error estimate shall be
established for the velocity and pressure approximations. In Section
\ref{Section:numerics}, some numerical experiments are conducted to
demonstrate the reliability, flexibility and accuracy of the weak
Galerkin method for the Brinkman equations. In particular, the first
example, which has known analytical solution, is designed to
demonstrate uniform convergence of the WG method with respect to
certain parameters. The rest of the examples are relevant to
practical problems for which no analytical solutions are known. In
addition, flow through different geometries are investigated in the
numerical experiments. These geometries include vuggy structure,
open foam and fibrous materials. Figure \ref{media} depicts the
profile of the permeability inverse for three highly varying porous
media under the present study.

\begin{figure}[!htb]\label{media}
\centering
\begin{tabular}{ccc}
  \resizebox{1.2in}{1.2in}{\includegraphics{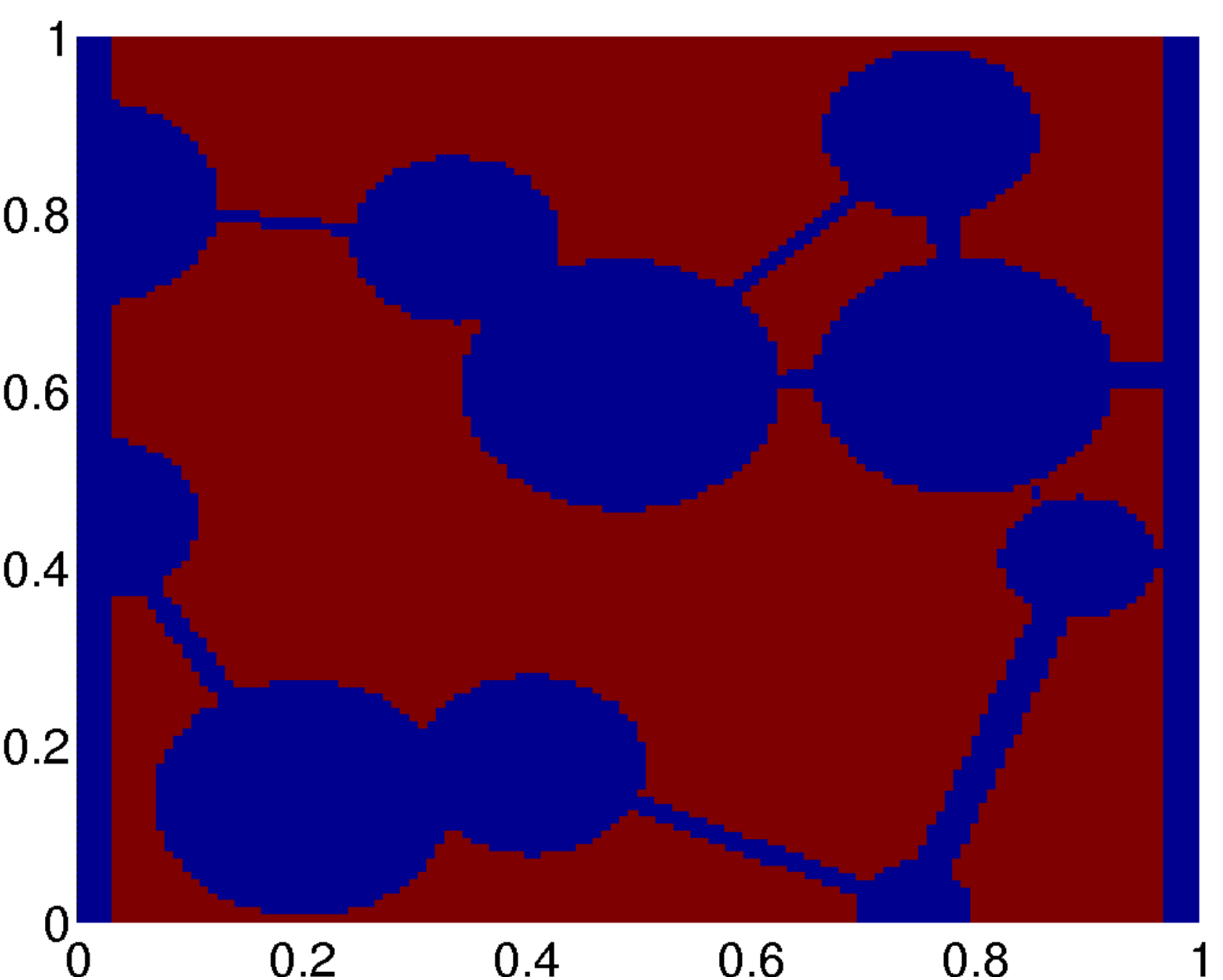}}&
  \resizebox{1.2in}{1.2in}{\includegraphics{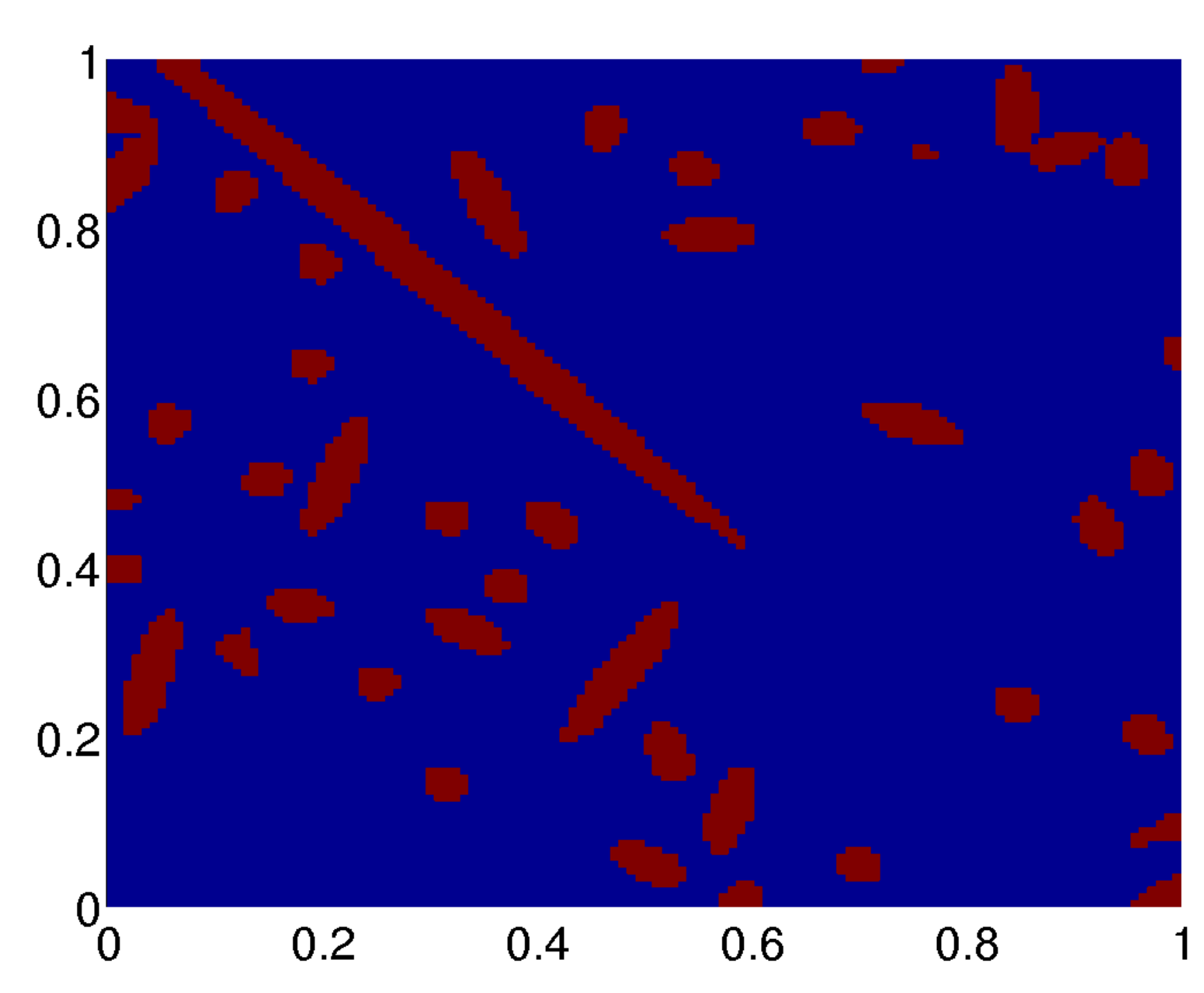}} &
  \resizebox{1.2in}{1.2in}{\includegraphics{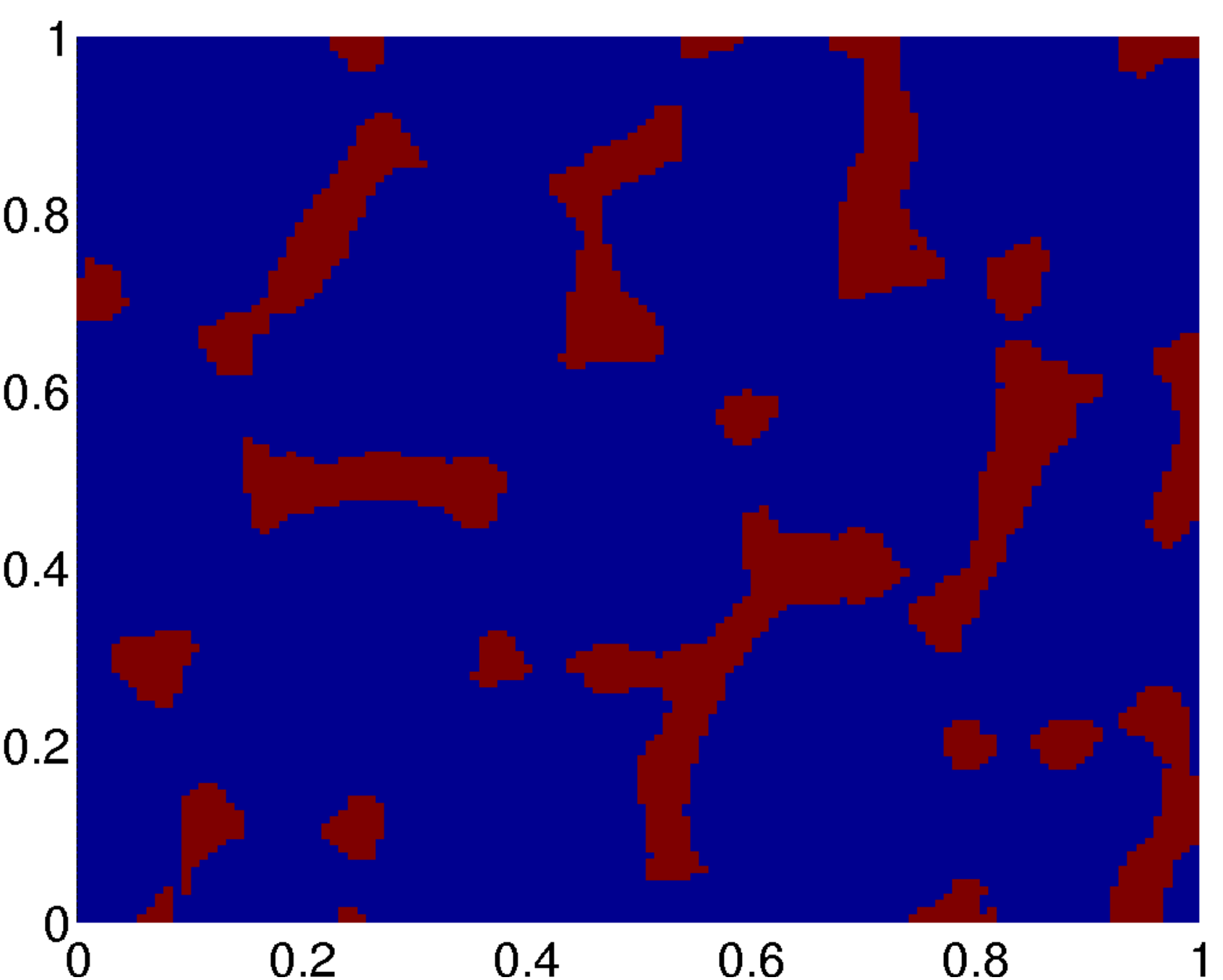}} \\
  (a) & (b)  & (c)
\end{tabular}
\caption{(a) vuggy medium; (b) fibrous material; (c) open foam.}
\end{figure}

\section{A Weak Galerkin Finite Element Method}\label{Section:section2}
First, we use the standard definition for the Sobolev space $H^s(D)$
and their associated inner products $(\cdot,\cdot)_{s,D}$, norms
$\|\cdot\|_{s,D}$, and seminorms $|\cdot|_{s,D}$ for any $s\ge 0$.
We shall drop the subscript $D$ when $D=\Omega$ and $s$ as $s=0$ in
the norm and inner product notation.

Let ${\cal T}_h$ be a partition of the domain $\Omega$ consisting of
polygons in two dimension or polyhedra in three dimension satisfying
a set of conditions as specified in \cite{wy-mixed}. Denote by
${\cal E}_h$ the set of all flat faces in ${\cal T}_h$, and let
${\cal E}_h^0={\cal E}_h\backslash\partial\Omega$ be the set of all
interior faces.

For $k\ge 1$, we define two weak Galerkin finite element spaces
associated with $\T_h$ as follows. For the velocity unknown, we have
\begin{equation}\label{FSVelocity}
V_h =\left\{ \bv=\{\bv_0, \bv_b\}:\ \{\bv_0, \bv_b\}|_{T}\in
[P_{k}(T)]^d\times [P_{k}(e)]^d,\ e\in\pT,\; \bv_b=0\ \mbox{on}\
\partial\Omega \right\},
\end{equation}
and for pressure
\begin{equation}\label{FSPressure}
W_h =\left\{q\in L_0^2(\Omega): \ q|_T\in P_{k-1}(T)\right\}.
\end{equation}
By a weak function $\bv=\{\bv_0, \bv_b\}$ we mean $\bv=\bv_0$ inside
of the element $T$ and $\bv=\bv_b$ on the boundary of the element
$T$. We would like to emphasize that any function $\bv\in V_h$ has a
single value $v_b$ on each edge $e\in\E_h$.

Our weak Galerkin finite element method is based on the following
variational formulation for (\ref{moment})-(\ref{bc}): find $(\bu,
p)\in [H_0^1(\Omega)]^d\times L_0^2(\Omega)$ satisfying
\begin{eqnarray}
(\nabla\bu,\nabla\bv)+(\A\bu,\bv)-(\nabla\cdot\bv,p)&=&({\bf f}, \bv),\label{w1}\\
(\nabla\cdot\bu,q)&=&0\label{w2}
\end{eqnarray}
for all $(\bv, q)\in [H_0^1(\Omega)]^d\times L_0^2(\Omega)$.

The key in the design of WG finite element scheme is the use of weak
derivatives in the place of strong derivatives in the variational
form for the underlying partial differential equations. Note that
the two differential operators used in (\ref{w1}) and (\ref{w2}) are
the gradient and divergence operators. Weak gradient and weak
divergence operators, along with their discrete analogues, have been
defined in \cite{wy} and \cite{wy-mixed} respectively. For
completeness, we recall the discrete weak divergence and weak
gradient operators as follows. For each $\bv=\{\bv_0, \bv_b\} \in
V_h$, the discrete weak divergence $\nabla_{w,k-1}\cdot\bv\in
P_{k-1}(T)$ is given on each element $T$ such that
\begin{equation}\label{d-d}
( \nabla_{w,k-1}\cdot\bv,q)_T = -(\bv_0,\nabla q)_T
+\l\bv_b,q\bn\r_\pT,\quad\quad\forall q\in P_{k-1}(T).
\end{equation}
Similarly, the discrete weak gradient $\nabla_{w,k-1}\bv\in
P_{k-1}(T)^{d\times d}$ is defined on each element $T$ by
\begin{equation}\label{d-g}
(\nabla_{w,k-1}\bv,\tau)_T = -(\bv_0,\nabla\cdot \tau)_T+ \l\bv_b,
\tau\cdot\bn \r_\pT,\quad\quad\forall \tau\in [P_{k-1}(T)]^{d\times
d}.
\end{equation}
Without confusion, we will drop the subscript $k-1$ and use
$\nabla_w\cdot$ and $\nabla_w$ to denote $\nabla_{w,k-1}\cdot$ and
$\nabla_{w,k-1}$. We will also use $(\nabla_w\cdot\bv,q)$ and
$(\nabla_w\bv,\nabla_w\bw)$ to denote
$\sum_{T\in\T_h}(\nabla_w\cdot\bv,q)_T$ and
$\sum_{T\in\T_h}(\nabla_w\bv,\nabla_w\bw)_T$ respectively.

We are now in a position to describe a weak Galerkin finite element
method for the Brinkman equations (\ref{moment})-(\ref{bc}). To this
end, we introduce three bilinear forms as follows
\begin{eqnarray*}
s(\bv,\bw) & = &\sum_{T\in {\cal T}_h}h_T^{-1}\l
\bv_0-\bv_b,\bw_0-\bw_b\r_\pT,\\
a(\bv,\bw)&=&(\nabla_w\bv,\nabla_w\bw)+(\A\bv_0,\bw_0)+s(\bv,\bw),\\
b(\bv,q)&=&(\nabla_w\cdot\bv,q).\\
\end{eqnarray*}

\begin{algorithm}
Find $\bu_h=\{\bu_0,\bu_b\}\in V_h$ and $p_h\in W_h$ such that
\begin{eqnarray}
a(\bu_h,\bv)-b(\bv,p_h)&=&(f,\bv_0),\qquad \forall\; \bv=\{\bv_0,\bv_b\}\in V_h\label{wg1}\\
b(\bu_h,q)&=&0, \;\qquad\forall\; q\in W_h .\label{wg2}
\end{eqnarray}
The corresponding solution $(\bu_h; p_h)$ is called WG finite
element solution for (\ref{moment})-(\ref{bc}).
\end{algorithm}
\smallskip

\section{Existence and Uniqueness}\label{Section:stability}

The WG finite element scheme (\ref{wg1})-(\ref{wg2}) is a
saddle-point problem. However, the theory of Babu\u{s}ka
\cite{babuska} and Brezzi \cite{brezzi} is hard to apply directly
due to the large variation of the permeability tensor. But the main
ideas of Babu\u{s}ka and Brezzi are still applicable.

For the velocity space $V_h$, we use a norm $\3bar\cdot\3bar$
induced by the symmetric an positive bilinear form $a(\cdot,\cdot)$
defined as follows
\begin{equation}\label{3barnorm}
\3bar
\bv\3bar^2=a(\bv,\bv)=\|\kappa^{-\frac12}\bv_0\|^2+\|\nabla_w\bv\|^2+\sum_{T\in\T_h}h_T^{-1}\|\bv_0-\bv_b\|_{\partial
T}^2.
\end{equation}
For convenience, we introduce another norm $\|\cdot\|_{1,h}$ in
$V_h$
\begin{equation}\label{Vh-anothernorm}
\|\bv\|_{1,h}^2
=\|\nabla_w\bv\|^2+\sum_{T\in\T_h}h_T^{-1}\|\bv_0-\bv_b\|_{\partial
T}^2.
\end{equation}
It is not hard to see that $\|\cdot\|_{1,h}$ is a discrete $H^1$
norm for $V_h$.

For the pressure space $W_h$, we use the following norm
\begin{equation}\label{H1norm}
|q|_{1,h}^2 = \sum_{T\in \T_h} \|\kappa^{\frac12}\nabla q\|_T^2
+h^{-1} \sum_{e\in\E_h^0} \|\jump{q}\|_e^2,
\end{equation}
where $\jump{q}$ is the jump of the function $q$ on the set of
interior edges $\E_h^0$.

For simplicity of analysis, the rest of the paper assumes that the
permeability tensor $\kappa$ has constant value on each element
$T\in\T_h$. The result can be easily extended to the case of
piecewise smooth tensor $\kappa$.

\smallskip
The following result is straightforward by using the definition of
$\3bar\cdot\3bar$ and the usual Cauchy-Schwarz inequality.

\begin{lemma}\label{coercivity+boundedness}
For any $\bv,\bw\in V_h$, we have the following boundedness and
coercivity for the bilinear form $a(\cdot,\cdot)$
\begin{eqnarray}
|a(\bv,\bw)|&\le& \3bar\bv\3bar\3bar\bw\3bar,\label{bd}\\
a(\bv,\bv)&=&\3bar\bv\3bar^2.\label{elliptic}
\end{eqnarray}
\end{lemma}

\medskip

For any $\rho\in W_h\subset L_0^2(\Omega)$ and $\bv\in V_h$, we have
from the definition of the discrete weak divergence that
\begin{eqnarray*}
b(\bv,\rho)&=&\sum_{T\in \T_h} (\nabla_w\cdot\bv, \rho)_T\\
&=&\sum_{T\in \T_h} \{\langle \bv_b,
\rho\bn\rangle_{\partial T}-(\bv_0, \nabla\rho)_T\}\\
&=&-\sum_{T\in \T_h} (\bv_0, \nabla\rho)_T+ \sum_{e\in\E_h^0}
\langle \bv_b, \jump{\rho}\bn_e\rangle_{e},
\end{eqnarray*}
where $\bn_e$ is a prescribed normal direction to the edge $e$, and
$\jump{\rho}$ stands for the jump of the function $\rho$ on edge
$e$. In particular, if $\bv=\bv^*=\{\bv_0^*, \bv_b^*\}$ is given by
\begin{equation}\label{CO-8}
\bv_0^*= -\kappa \nabla\rho,\quad \bv_b^* = h^{-1}\jump\rho\bn_e,
\end{equation}
then
\begin{equation}\label{CO.01}
b(\bv^*,\rho) =\sum_{T\in \T_h} (\kappa\nabla\rho,
\nabla\rho)_T+h^{-1} \sum_{e\in\E_h^0}
\|\jump{\rho}\|_{e}^2=|\rho|^2_{1,h}.
\end{equation}
Thus, $\bv^*$ can be regarded as an artificial flux for the
``pressure" function $\rho$. For convenience, we introduce a
notation for this artificial flux:
\begin{equation}\label{artificial-flux}
F(\rho):=\{-\kappa\nabla\rho,\ h^{-1}\jump\rho\bn_e\}.
\end{equation}

\smallskip
\begin{lemma}\label{Lemma:inf-sup}
For any $\rho\in\W_h$, let $F(\rho)$ be the artificial flux given by
(\ref{artificial-flux}). Then, we have
\begin{equation}\label{CO.0101}
b(F(\rho),\rho)=|\rho|^2_{1,h}.
\end{equation}
Furthermore, there exists a constant $C$ such that
\begin{equation}\label{CO.0102}
\|F(\rho)\|_{1,h} \leq C h^{-1} |\rho|_{1,h}.
\end{equation}
\end{lemma}

\begin{proof}
The identity (\ref{CO.0101}) is given by (\ref{CO.01}). It remains
to derive the estimate (\ref{CO.0102}). To this end, write
$\{\bv_0^*, \bv_b^*\}=F(\rho)$. From the definition
(\ref{Vh-anothernorm}), we have
\begin{equation}\label{3barnorm:CO-100}
\|F(\rho)\|^2_{1,h} =\|\nabla_w\bv^*\|^2+
\sum_{T\in\T_h}h_T^{-1}\|\bv_0^*-\bv_b^*\|_{\partial T}^2.
\end{equation}
To estimate the first term $\|\nabla_w\bv^*\|^2$, we recall from the
definition of $\nabla_w\bv^*$ that
$$
(\nabla_w\bv^*,\tau)_T = - (\bv_0^*, \nabla\cdot\tau)_T + \langle
\bv_b^*, \tau\cdot\bn\rangle_{\partial T}.
$$
Using (\ref{artificial-flux}) we obtain
$$
(\nabla_w\bv^*,\tau)_T = (\kappa\nabla\rho, \nabla\cdot\tau)_T +
h^{-1}\langle \jump{\rho}\bn, \tau\cdot\bn\rangle_{\partial
T\cap\E_h^0}.
$$
The above equation, together with the usual inverse inequality for
finite element functions implies
$$
\|\nabla_w\bv^*\|_T \leq C h^{-1} \|\kappa\nabla\rho\|_T + C
h^{-\frac32}\|\jump{\rho}\|_{\partial T\cap\E_h^0}.
$$
Thus, from the assumption (\ref{ellipticity}) we have
$$
\|\nabla_w\bv^*\|_T^2 \leq C h^{-2} \lambda_1^{-1}
\|\kappa^{\frac12}\nabla\rho\|_T^2 + C
h^{-3}\|\jump{\rho}\|_{\partial T\cap\E_h^0}^2.
$$
Summing over all element $T\in\T_h$ yields
\begin{eqnarray}\nonumber
\|\nabla_w\bv^*\|^2 &\leq& Ch^{-2}\lambda_1^{-1} \sum_{T\in \T_h}
(\kappa\nabla\rho, \nabla\rho)_T+h^{-3} \sum_{e\in\E_h^0}
\|\jump{\rho}\|_{e}^2\\
&\leq& C h^{-2} |\rho|^2_{1,h},\label{CO-200}
\end{eqnarray}
where we have used the assumption that $\lambda_1$ is of unit size.
As to the second term on the right-hand side of
(\ref{3barnorm:CO-100}), we use (\ref{artificial-flux}) and the
trace inequality (\ref{trace}) to obtain
\begin{eqnarray*}
\|\bv_0^*-\bv_b^*\|_{\partial T}^2&\leq& 2
\|\kappa\nabla\rho\|^2_{\partial T}+2
h^{-2}\|\jump{\rho}\|^2_{\partial T\cap\E_h^0}\\
&\leq& C (h^{-1}\|\kappa\nabla\rho\|^2_{T}+
h^{-2}\|\jump{\rho}\|^2_{\partial T\cap\E_h^0}).
\end{eqnarray*}
Thus, it follows from (\ref{ellipticity}) that
\begin{eqnarray}
\sum_{T\in\T_h} h_T^{-1}\|\bv_0^*-\bv_b^*\|_{\partial T}^2 &\leq& C
h^{-2}\lambda_1^{-1}\sum_{T\in\T_h}
\|\kappa^{\frac12}\nabla\rho\|_T^2 +
Ch^{-3}\sum_{e\in\E_h^0}\|\jump{\rho}\|^2_{e}\nonumber\\
&\leq& C h^{-2} |\rho|_{1,h}^2.\label{CO-300}
\end{eqnarray}
Here we again used the fact that $\lambda_1$ is of unit size.
Substituting (\ref{CO-200}) and (\ref{CO-300}) into
(\ref{3barnorm:CO-100}) yields the desired estimate (\ref{CO.0102}).
This completes the proof of the lemma.
\end{proof}

\medskip
\begin{lemma}
The weak Galerkin finite element scheme (\ref{wg1})-(\ref{wg2}) has
one and only one solution.
\end{lemma}

\begin{proof}
Since the number of unknowns is the same as the number of equations,
then the solution existence is equivalent to its uniqueness. Thus,
it suffices to show that the homogeneous problem (i.e., $f=0$) has
only trivial solutions. To this end, assume that $f=0$ in
(\ref{wg1}). By letting $\bv=\bu_h$ in (\ref{wg1}) and $q=p_h$ in
(\ref{wg2}) we obtain
$$
a(\bu_h, \bu_h) - b(\bu_h, p_h) = 0,\quad b(\bu_h,p_h)=0.
$$
It follows that
$$
\3bar\bu_h\3bar = \sqrt{a(\bu_h, \bu_h)}=0,
$$
and hence $u_h=0$.

To show $p_h=0$, we use the equation (\ref{wg1}) and the face that
$f=0$ and $\bu_h=0$ we obtain
$$
b(\bv, p_h)=0.
$$
By letting $\bv=F(p_h)$ be the artificial flux of $p_h$, we have
from (\ref{CO.0101}) that
$$
0=b(F(p_h), p_h)= |p_h|_{1,h}.
$$
Thus, $p_h=0$ and the lemma is completely proved.
\end{proof}

\section{Error Equations}\label{Section:errorequations}

Denote by $Q_{0}$ the $L^2$ projection operator from $[L^2(T)]^d$
onto $[P_k(T)]^d$. For each edge/face $e\in {\cal E}_h$, denote by
$Q_{b}$ the $L^2$ projection from $[L^2(e)]^d$ onto $[P_{k}(e)]^d$.
We shall combine $Q_0$ with $Q_b$ by writing $Q_h=\{Q_0,Q_b\}$. In
addition, let $\bbQ_h$ and ${\bf Q}_h$ be two local $L^2$
projections onto $P_{k-1}(T)$ and $[P_{k-1}(T)]^{d\times d}$,
respectively.

\begin{lemma}\label{lem-1-0}
The projection operators $Q_h$, ${\bf Q}_h$, and $\bbQ_h$ satisfy
the following commutative properties
\begin{eqnarray}
\nabla_w (Q_h \bv) &=& {\bf Q}_h (\nabla\bv),\qquad\forall \ \bv\in
[H^1(\Omega)]^d,\label{key}\\
\label{key11} \nabla_w\cdot (Q_h\bv)
&=&\bbQ_h(\nabla\cdot\bv),\qquad\forall\ \bv\in H({\rm
div},\Omega).\label{key1}
\end{eqnarray}
\end{lemma}
The proof of Lemma \ref{lem-1-0} is straightforward and can be found
in \cite{wy} and \cite{wy-mixed}.

\smallskip
The following are two useful identities:
\begin{eqnarray}
(\nabla_w(Q_h\bu),\nabla_w\bv)_T&=&(\nabla\bu,\nabla\bv_0)_T-\l \bv_0-\bv_b,{\bf Q}_h(\nabla\bu)\cdot\bn\r_\pT.\label{m1}\\
(\bv_0,\nabla p)&=&-(\dw\bv,\bbQ_hp)+\sum_{T\in\T_h}\langle
\bv_0-\bv_b,(p-\bbQ_hp)\bn\rangle_\pT.\label{m2}
\end{eqnarray}
Equations (\ref{m1}) and (\ref{m2}) can be verified easily; they
were first derived in \cite{wy} and \cite{wy-mixed}, respectively.

\smallskip
Introduce two functionals as follows
\begin{eqnarray}
l_1(\bv,\bu)&=&\sum_{T\in\T_h}\l\bv_0-\bv_b,\nabla\bu\cdot\bn-{\bf Q}_h(\nabla\bu)\cdot\bn\r_\pT,\label{l1}\\
l_2(\bv,p)&=&\sum_{T\in\T_h}\langle
\bv_0-\bv_b,(p-\bbQ_hp)\bn\rangle_\pT.\label{l2}
\end{eqnarray}

\begin{lemma}\label{Lemma:4.1}
Let $\bu_h=\{\bu_0,\bu_b\}$ be the WG finite element solution
arising from the {\sc Weak Galerkin Algorithm 1}. Let
$\be_h=\{\be_0,\;\be_b\}=\{Q_0\bu-\bu_0,\;Q_b\bu-\bu_b\}$ and
$\varepsilon_h=\bbQ_hp-p_h$ be the error between the WG finite
element solution and the $L^2$ projection of the exact solution.
Then, the following equations are satisfied
\begin{eqnarray}
a(\be_h,\bv)-b(\bv,\varepsilon_h)&=& \phi_{\bu,p}(\bv),\label{ee1}\\
b(\be_h,q)&=&0,\label{ee2}
\end{eqnarray}
for all $(\bv; q)\in V_h\times W_h$. Here
\begin{equation}
\phi_{\bu,p}(\bv)=l_1(\bv,\bu)-l_2(\bv,p)+s(Q_h\bu,\bv).\label{phi-up}
\end{equation}
\end{lemma}

\begin{proof} Testing (\ref{moment}) by $\bv_0$ with
$\bv=\{\bv_0,\;\bv_b\}\in V_h$ gives
\begin{equation}\label{m3}
-(\Delta\bu,\bv_0)+(\A\bu, \bv_0)+(\nabla p,\bv_0)=(\bbf, \bv_0).
\end{equation}
It follows from the integration by parts that
\[
-(\Delta\bu,\bv_0)=\sum_{T\in\T_h}(\nabla\bu,\nabla\bv_0)_T
-\sum_{T\in\T_h}\l\bv_0-\bv_b,\nabla\bu\cdot\bn\r_\pT,
\]
where we have used the fact that $\sum_{T\in\T_h}\langle\bv_b,
\nabla\bu\cdot\bn\rangle_\pT=0$. Using (\ref{m1}) and the equation
above, we obtain
\begin{eqnarray}
-(\Delta\bu,\bv_0)=(\nabla_w
(Q_h\bu),\nabla_w\bv)-l_1(\bv,\bu).\label{m4}
\end{eqnarray}
Using (\ref{m2}), (\ref{m4}) and the definition of $Q_0$, we have
\begin{eqnarray}
&&-(\Delta\bu,\bv_0)+(\A\bu, \bv_0)+(\nabla p,\bv_0)=(\nabla_w (Q_h\bu),\nabla_w\bv)+
(\A\bu,\bv_0)\nonumber\\
&&\hskip .5in -(\dw\bv,\bbQ_hp)- l_1(\bv,\bu)+l_2(\bv,p)\label{mmm}
\end{eqnarray}
It follows from (\ref{m3}) and (\ref{mmm}),
\[
(\nabla_w(Q_h\bu),\nabla_w\bv)+(\A Q_0\bu,\bv_0)-(\dw\bv,\bbQ_hp)
=(\bbf,\bv_0)+ l_1(\bv,\bu)-l_2(\bv,p).
\]
Adding $s(Q_h\bu,\bv)$ to the both sides of the equation above gives
\begin{equation}\label{m5}
a(Q_h\bu,\bv)-b(\bv,\bbQ_hp)=(\bbf,\bv_0)+ \phi_{\bu,p}(\bv).
\end{equation}
The difference of (\ref{m5}) and (\ref{wg1}) yields the following equation,
\begin{eqnarray}
a(\be_h,\bv)-b(\bv,\varepsilon_h)= \phi_{\bu,p}(\bv) \label{eee1}
\end{eqnarray}
for all $\bv\in V_h$. Next, testing Equation (\ref{cont}) by $q\in
W_h$ and using (\ref{key11}) gives
\begin{equation}\label{m6}
(\nabla\cdot\bu,q)=(\nabla_w\cdot Q_h\bu,q)=0.
\end{equation}
The difference of (\ref{m6}) and (\ref{wg2}) yields the following
equation.
\begin{eqnarray}
b(\be_h,q)=0,\quad\forall q\in W_h.\label{eee2}
\end{eqnarray}
Combining (\ref{eee1}) and (\ref{eee2}) completes the proof of the
lemma.
\end{proof}

\section{Preparation for Error Estimates}

In this section, we will derive some estimates that can be used in
the next section to obtain uniform convergence for velocity and
pressure approximations.

The following lemma provides some approximation properties for the
projections $Q_h$, ${\bf Q}_h$ and $\bbQ_h$.  Observe that the
underlying mesh $\T_h$ is assumed to be sufficiently general to
allow polygons or polyhedra. A proof of the lemma can be found in
\cite{wy-mixed}.

\begin{lemma}\label{lem-1}
Let $\T_h$ be a finite element partition of $\Omega$ satisfying the
shape regularity assumption as specified in \cite{wy-mixed} and
$\bw\in [H^{r+1}(\Omega)]^d$ and $\rho\in H^{r}(\Omega)$ with $1\le
r\le k$. Then, for $0\le s\le 1$ we have
\begin{eqnarray}
&&\sum_{T\in\T_h} h^{2s}_T\|\bw-Q_0\bw\|_{T,s}^2\le h^{2(r+1)}
\|\bw\|^2_{r+1},\label{Qh}\\
&&\sum_{T\in\T_h} h^{2s}_T\|\nabla\bw-{\bf Q}_h(\nabla\bw)\|^2_{T,s}
\le Ch^{2r}
\|\bw\|^2_{r+1},\label{Rh}\\
&&\sum_{T\in\T_h} h^{2s}_T\|\rho-\bbQ_h\rho\|^2_{T,s} \le
Ch^{2r}\|\rho\|^2_{r}.\label{Lh}
\end{eqnarray}
Here $C$ denotes a generic constant independent of the meshsize $h$
and the functions in the estimates.
\end{lemma}
\smallskip

Let $T$ be an element with $e$ as a face.  For any function $g\in
H^1(T)$, the following
trace inequality has been proved to be valid for general meshes described in
\cite{wy-mixed},
\begin{equation}\label{trace}
\|g\|_{e}^2 \leq C \left( h_T^{-1} \|g\|_T^2 + h_T \|\nabla
g\|_{T}^2\right).
\end{equation}
For any finite element function $\bv\in V_h$, we introduce the
following semi-norm:
\begin{equation}\label{new-seminorm}
|v|_h=\left(\sum_{T\in\T_h}h_T^{-1}\|\bv_0-\bv_b\|^2_{\pT}\right)^{1/2}.
\end{equation}

\begin{lemma}\label{lemma5.2}
Let $r\in [1,k]$. Assume that $\bw\in [H^{r+1}(\Omega)]^d$ and
$\rho\in H^r(\Omega)$. Then for any $\bv\in V_h$ we have
\begin{eqnarray}
|s(Q_h\bw,\bv)|&\le& Ch^{r}\|\bw\|_{r+1} |\bv|_h,\label{mmm1}\\
|l_1(\bv,\bw)|&\le& Ch^{r}\|\bw\|_{r+1} |\bv|_h,\label{mmm2}\\
|l_2(\bv,\rho)|&\le& Ch^{r}\|\rho\|_{r} |\bv|_h,\label{mmm3}
\end{eqnarray}
where $l_1(\cdot,\cdot)$ and $l_2(\cdot,\cdot)$ are defined in
(\ref{l1}) and (\ref{l2}). Thus, the following estimate holds true
\begin{equation}\label{phi-up-estimate}
|\phi_{\bw,\rho}(\bv)|\leq C h^r (\|\bw\|_{r+1}+\|\rho\|_{r})|
\bv|_h.
\end{equation}
\end{lemma}

\begin{proof}
Using the definition of $Q_b$, (\ref{trace}) and (\ref{Qh}), we have
\begin{eqnarray*}
|s(Q_h\bw,\bv)|&=&\left|\sum_{T\in\T_h} h_T^{-1}\langle Q_0\bw-Q_b\bw, \bv_0-\bv_b\rangle_\pT\right|\\
&=&\left|\sum_{T\in\T_h} h_T^{-1} \langle Q_0\bw-\bw, \bv_0-\bv_b\rangle_\pT\right|\\
&\le& \left(\sum_{T\in\T_h}(h_T^{-2}\|Q_0\bw-\bw\|_T^2+\|\nabla (Q_0\bw-\bw)\|_T^2)\right)^{1/2}
|\bv|_h\\
&\le& Ch^{r}\|\bw\|_{r+1} |\bv|_h.
\end{eqnarray*}
Similarly, it follows from (\ref{trace}) and (\ref{Rh})
\begin{eqnarray*}
|l_1(\bv,\bw)|&\equiv&\left|\sum_{T\in\T_h}\l\bv_0-\bv_b,\nabla\bw\cdot\bn-{\bf Q}_h(\nabla\bw)\cdot\bn\r_\pT\right|\\
&\le& \left(\sum_{T\in\T_h}h\|\nabla\bw\cdot\bn-{\bf
Q}_h(\nabla\bw)\cdot\bn\|_\pT^2\right)^{1/2}
|\bv|_h\\
&\le& Ch^{r}\|\bw\|_{r+1} |\bv|_h.
\end{eqnarray*}
Using (\ref{trace}) and (\ref{Lh}), we have
\begin{eqnarray*}
|l_2(\bv,\rho)|&\equiv&\left|\sum_{T\in\T_h}\langle \bv_0-\bv_b,(\rho-\bbQ_h\rho)\bn\rangle_\pT
\right|\\
&\le& \left(\sum_{T\in\T_h}h_T\| \rho-\bbQ_h\rho\|_\pT^2\right)^{1/2}
|\bv|_h\\
&\le& Ch^{r}\|\rho\|_{r} |\bv|_h.
\end{eqnarray*}
This completes the proof of the lemma.
\end{proof}

\section{Error Estimates}\label{Section:errores}

The goal of this section is to establish some error estimates for
the approximate velocity $\bu_h$ in the triple-bar norm and for the
approximate pressure $p_h$ in the usual $L^2$ norm. Our main result
can be stated as follows.
\smallskip

\begin{theorem}\label{h1-bd}
Let $(\bu; p)\in  [H_0^1(\Omega)\cap H^{k+1}(\Omega)]^d\times
(L_0^2(\Omega)\cap H^{k}(\Omega))$ with $k\ge 1$ and $(\bu_h;p_h)\in
V_h\times W_h$ be the solutions of (\ref{moment})-(\ref{bc}) and
(\ref{wg1})-(\ref{wg2}) respectively. Then, there exists a constant
$C$ independent of the meshsize $h$ and the spectral radius of
$\kappa$ such that
\begin{eqnarray}
\3bar \bu-\bu_h\3bar+h|p-p_h|_{1,h}&\le&
Ch^{k}(\|\bu\|_{k+1}+\|p\|_{k}).\label{err1}
\end{eqnarray}
In particular, we have the following weighted-$L^2$ error estimate:
\begin{eqnarray}
\|\kappa^{-\frac12}(\bu-\bu_h)\|&\le&
Ch^{k}(\|\bu\|_{k+1}+\|p\|_{k}).\label{L2err1}
\end{eqnarray}
\end{theorem}

\smallskip

\begin{proof}
Letting $\bv=\be_h$ in (\ref{ee1}) and $q=\varepsilon_h$ in
(\ref{ee2}) and adding the two resulting equations, we obtain
\begin{eqnarray}
\3bar \be_h\3bar^2&=&\phi_{\bu,p}(\be_h).\label{main}
\end{eqnarray}
Using the estimate (\ref{phi-up-estimate}) with $r=k, \bw=\bu,$ and
$\rho=p$ we arrive at
\begin{equation}\label{b-u}
\3bar \be_h\3bar^2 \le Ch^{k}(\|\bu\|_{k+1}+\|p\|_{k}) |\be_h|_h.
\end{equation}
It is trivial to see that
$$
|\be_h|_h \leq C\3bar \be_h\3bar.
$$
Substituting the above estimate into (\ref{b-u}) yields the desired
error estimate for $\bu_h$.

To derive an estimate for $\varepsilon_h$, we have from (\ref{ee1})
that
\begin{equation}\label{Hi-001}
b(\bv,\varepsilon_h)=a(\be_h,\bv)- \phi_{\bu,p}(\bv),
\end{equation}
for all $\bv\in V_h$. In particular, by letting $\bv=
F(\varepsilon_h)$ be the artificial flux of $\varepsilon_h$, we have
from Lemma \ref{Lemma:inf-sup} that
\begin{eqnarray*}
|\varepsilon_h|_{1,h}^2 &=& b(F(\varepsilon_h), \varepsilon_h)\\
&=&a(\be_h, F(\varepsilon_h))- \phi_{\bu,p}(F(\varepsilon_h)).
\end{eqnarray*}
Thus, from the definition of $a(\cdot,\cdot)$ and $F(\varepsilon_h)$
we have
\begin{eqnarray*}
|\varepsilon_h|_{1,h}^2 &\leq& \|\be_h\|_{1,h}\ \|
F(\varepsilon_h)\|_{1,h} + |(\kappa^{-1}\be_0, \kappa\nabla
\varepsilon_h)|+ Ch^{k}(\|\bu\|_{k+1}+\|p\|_{k})
|F(\varepsilon_h)|_h\\
&\leq& \|\be_h\|_{1,h}\ \| F(\varepsilon_h)\|_{1,h} +
|(\kappa^{-\frac12}\be_0, \kappa^{\frac12}\nabla \varepsilon_h)|+
Ch^{k}(\|\bu\|_{k+1}+\|p\|_{k}) \|F(\varepsilon_h)\|_{1,h}\\
&\leq& Ch^{k}(\|\bu\|_{k+1}+\|p\|_{k}) \|F(\varepsilon_h)\|_{1,h} +
\|\kappa^{-\frac12} \be_0\|\ \|\kappa^{\frac12}\nabla
\varepsilon_h\|\\
&\leq& Ch^{k}(\|\bu\|_{k+1}+\|p\|_{k}) h^{-1} |\varepsilon_h|_{1,h}
+ \|\kappa^{-\frac12} \be_0\|\ |\varepsilon_h|_{1,h}.
\end{eqnarray*}
Dividing $h^{-1}|\varepsilon_h|_{1,h}$ from both sides of the above
inequality leads to
$$
h |\varepsilon_h|_{1,h} \leq Ch^{k}(\|\bu\|_{k+1}+\|p\|_{k}) + C
h\|\kappa^{-\frac12} \be_0\|.
$$
This completes the proof of (\ref{err1}).
\end{proof}

\medskip

The rest of this section is devoted to an error estimate for the
velocity approximation in the standard $L^2$ norm by following the
routine duality argument. The analysis is very much along the same
line as for the Stokes equation \cite{wg-stokes}. More precisely,
let us consider the dual problem which seeks $(\psi;\xi)$ satisfying
\begin{eqnarray}
-\Delta\psi+\A\psi+\nabla \xi &=&\be_0\quad \mbox{in}\;\Omega,\label{dual-m}\\
\nabla\cdot\psi&=&0\quad\mbox{in}\;\Omega,\label{dual-c}\\
\psi&=& 0\quad\mbox{on}\;\partial\Omega.\label{dual-bc}
\end{eqnarray}
Assume that the dual problem has the $[H^{2}(\Omega)]^d\times
H^1(\Omega)$-regularity property in the sense that the solution
$(\psi; \xi)\in [H^{2}(\Omega)]^d\times H^1(\Omega)$ and the
following a priori estimate holds true:
\begin{equation}\label{reg}
\|\psi\|_{2}+\|\xi\|_1\le C\|\be_0\|.
\end{equation}
The assumption (\ref{reg}) is known to be valid when the domain
$\Omega$ is convex and the permeability tensor $\kappa$ is not
highly varying.

\medskip
\begin{theorem}
Let $(\bu; p)\in  [H_0^1(\Omega)\cap H^{k+1}(\Omega)]^d\times
L^2_0(\Omega)\cap H^{k}(\Omega)$ with $k\ge 1$ and $(\bu_h;p_h)\in
V_h\times W_h$ be the solutions of (\ref{moment})-(\ref{bc}) and
(\ref{wg1})-(\ref{wg2}) respectively. Assume that (\ref{reg}) holds
true. Then one has the following estimate
\begin{equation}\label{l2-error}
\|\bu-\bu_0\|\le Ch^{k+1}(\|\bu\|_{k+1}+\|p\|_{k}).
\end{equation}
\end{theorem}

\begin{proof}
Testing (\ref{dual-m}) by $\be_0$ gives
\[
\|Q_0\bu-\bu_0\|^2=(\be_0,\be_0)=-(\Delta\psi,
\be_0)+(\A\psi,\be_0)+(\nabla\xi,\be_0).
\]
Using (\ref{mmm}) with $\bu=\psi$, $\bv_0=\be_0$ and $p=\xi$, the above equation becomes
\begin{eqnarray*}
\|Q_0\bu-\bu_0\|^2&=&(\nabla_wQ_h\psi,
\nabla_w\be_h)+(\psi,\A\be_0)-(\nabla_w\cdot\be_h,\bbQ_h\xi)\\
& & - l_1(\be_h,\psi)+l_2(\be_h,\xi)\\
&=&(\nabla_wQ_h\psi,
\nabla_w\be_h)+(Q_0\psi,\A\be_0)-(\nabla_w\cdot\be_h,\bbQ_h\xi)\\
& & - l_1(\be_h,\psi)+l_2(\be_h,\xi).\\
\end{eqnarray*}
Adding and subtracting $s(Q_h\psi,\be_h)$ in the equation above
yields
\begin{eqnarray*}
\|Q_0\bu-\bu_0\|^2&=&a(Q_h \psi,\be_h)-b(\be_h,\bbQ_h\xi)
-\phi_{\psi,\xi}(\be_h),
\end{eqnarray*}
where the functional $\phi_{\psi,\xi}$ is given as in
(\ref{phi-up}). Now using the fact that $b(\be_h, \bbQ_h\xi)=0$ and
$b(Q_h\psi,\varepsilon_h)=0$ we obtain
\begin{eqnarray*}
\|Q_0\bu-\bu_0\|^2&=&a(\be_h,Q_h\psi)-b(Q_h\psi,\varepsilon_h)-
\phi_{\psi,\xi}(\be_h).
\end{eqnarray*}
Using the first error equation (\ref{ee1}), we can rewrite the above
equation as follows
\begin{eqnarray}
\|Q_0\bu-\bu_0\|^2&=& \phi_{\bu,p} (Q_h\psi) -
\phi_{\psi,\xi}(\be_h).\label{d1}
\end{eqnarray}
The right-hand side of (\ref{d1}) can be bounded by using the
estimate (\ref{phi-up-estimate}). To this end, using
(\ref{phi-up-estimate}) with $r=k, \ \bw=\bu$, and $\rho=p$ we
obtain
\begin{equation}\label{yes.108}
|\phi_{\bu,p} (Q_h\psi)|\leq C h^k (\|\bu\|_{k+1}+\|p\|_k)\
|Q_h\psi|_h.
\end{equation}
Note that from the trace inequality (\ref{trace}) and the definition
of $Q_b$ we have
\begin{eqnarray*}
|Q_h\psi |_h^2&=&\sum_{T\in\T_h} h_T^{-1}\|Q_0\psi-Q_b\psi\|^2_\pT\nonumber\\
&\le& \sum_{T\in\T_h} h_T^{-1}\|Q_0\psi-\psi\|_\pT^2+\sum_{T\in\T_h} h_T^{-1}\|\psi-Q_b\psi\|_\pT^2\nonumber\\
&\le& C\sum_{T\in\T_h} h_T^{-1}\|Q_0\psi-\psi\|_\pT^2\le
Ch^{2}\|\psi\|_{2}^2.
\end{eqnarray*}
Substituting the above into (\ref{yes.108}) gives
\begin{equation}\label{yes.109}
|\phi_{\bu,p} (Q_h\psi)|\leq C h^{k+1} (\|\bu\|_{k+1}+\|p\|_k)\
\|\psi\|_2.
\end{equation}
Next, using (\ref{phi-up-estimate}) with $r=1, \ \bw=\psi$, and
$\rho=\xi$ we obtain
\begin{eqnarray}
|\phi_{\psi,\xi} (\be_h)|&\leq& C h (\|\psi\|_{2}+\|\xi\|_1)\
|\be_h|_h\nonumber\\
&\leq& C h (\|\psi\|_{2}+\|\xi\|_1)\3bar\be_h\3bar.\label{yes.209}
\end{eqnarray}
Substituting (\ref{yes.109}) and (\ref{yes.209}) into (\ref{d1})
yields
\begin{eqnarray*}
\|Q_0\bu-\bu_0\|^2&\leq & C h^{k+1} (\|\bu\|_{k+1}+\|p\|_k)\
\|\psi\|_2 + C h (\|\psi\|_{2}+\|\xi\|_1)\3bar\be_h\3bar.
\end{eqnarray*}
Finally, we apply the regularity estimate (\ref{reg}) to the above
estimate to obtain
\begin{eqnarray*}
\|Q_0\bu-\bu_0\|&\leq & C h^{k+1} (\|\bu\|_{k+1}+\|p\|_k) + C h
\3bar\be_h\3bar,
\end{eqnarray*}
which, together with the error estimate (\ref{err1}), completes the
proof of the lemma.
\end{proof}

\section{Numerical Experiments}\label{Section:numerics}

The goal of this section is to numerically demonstrate the
efficiency of the WG finite element algorithm
(\ref{wg1})-(\ref{wg2}) when the lowest order of element (i.e.,
$k=1$) is employed. For simplicity, we consider the Brinkman model
(\ref{moment})-(\ref{bc}) in two dimensional domains. The error for
the WG solution of (\ref{wg1})-(\ref{wg2}) is measured in three
norms defined as follows:

$$
\displaystyle\3bar \bv\3bar^2:
=\sum_{T\in\mathcal{T}_h}\left(\mu\int_T|\nabla_w\bv|^2dx+\int_T\mu\A\bv_0\cdot\bv_0dx
+\int_{\partial T}h^{-1}_T(\bv_0-\bv_b)^2ds\right).
$$
$$
\displaystyle\|\bv\|^2:=\sum_{T\in\mathcal{T}_h}\int_T|\bv_0|^2dx,
$$
$$
\displaystyle\|q\|^2:=\sum_{T\in\mathcal{T}_h}\int_T|q|^2dx.
$$
Note that $\3bar\bv\3bar$ is a discrete $H^1$ norm, and the other
two are the standard $L^2$ norm in the respective spaces. Here
$\nabla_w\bv\in [P_0(T)]^{2\times 2}$ is computed on each element
$T\in\T_h$ by the following equation
$$
(\nabla_w\bv,\tau)_T=-(\bv_0,\nabla\cdot\tau)_T+\l\bv_b,\tau\cdot{\bf
n}\r_{\partial T}
$$
for all $\tau\in [P_0(T)]^{2\times 2}.$ Since $\tau$ is constant on
the element $T$, the above equation can be simplified as
\begin{eqnarray*}
(\nabla_w \bv,\tau)_T=\l\bv_b,\tau\cdot{\bf n}\r_{\partial T},
\quad\forall\tau\in[P_0(T)]^{2\times 2}.
\end{eqnarray*}

For any given $\bv=\{\bv_0,\bv_b\}\in V_h$, the discrete weak
divergence $\nabla_w\cdot\bv\in P_0(T)$ on each element $T\in\T_h$
is computed by solving the following equation
$$
(\nabla_w\cdot \bv,q)_T=-(\bv_0,\nabla
q)_T+\l\bv_b\cdot\bn,q\r_{\partial T}, \quad\forall q\in P_0(T).
$$
Since $q|_T\in P_0(T),$ the above equation can be simplified as
\begin{eqnarray*}
(\nabla_w\cdot \bv,q)_T=\l\bv_b\cdot\bn,q\r_{\partial T},
\quad\forall q\in P_0(T).
\end{eqnarray*}

The numerical examples of this section have been considered in
\cite{efendiev2011robust,iliev2011variational,willems2009numerical}.
Examples 1, 2 and 3 are presented for studying the reliability of
the WG method for problems with high contrast of permeability such
that $\kappa^{-1}$ varies from 1 to $10^6$. In such geometry, large
highly permeable media connect vugs surrounded by a rather lowly
permeable material. Example 1 has known analytical solution (see
\cite{willems2009numerical}). But Examples 2 and 3 do not have
analytical solutions to the author's knowledge. The profiles of
$\kappa^{-1}$ for examples 2 and 3 can be found in
\cite{efendiev2011robust}.

\subsection{Example 1}
\begin{figure}[h!]
\centering
\begin{tabular}{c}
  \resizebox{3in}{2.5in}{\includegraphics{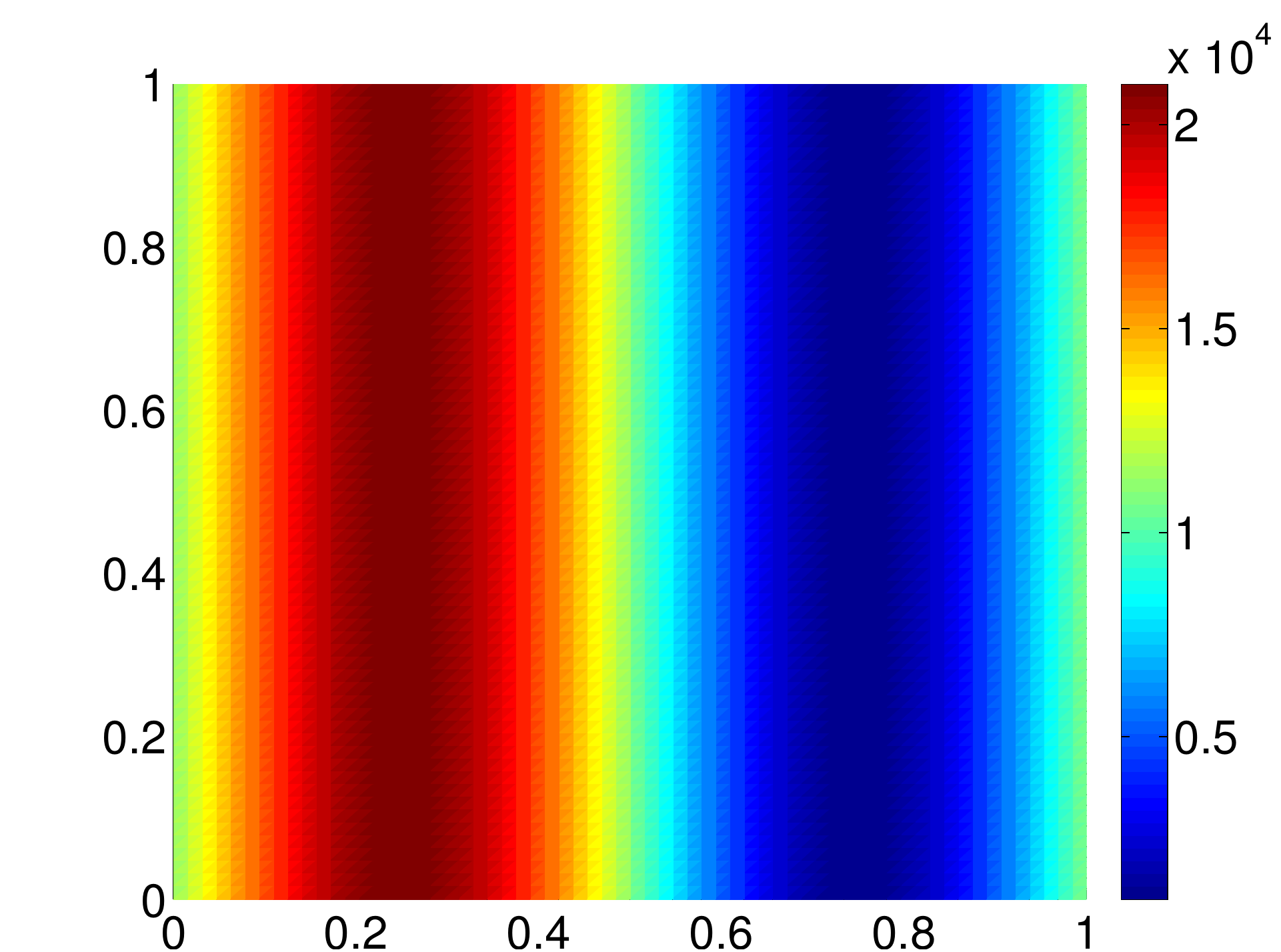}}
  \end{tabular}
\caption{Geometry for $\kappa^{-1}$ in Example 1 with
$a=10^4$.}\label{fig:ex3_kappa}
\end{figure}

This example will test the accuracy and reliability of the method
for a giving analytical solutions and highly varying permeability
$\kappa$. The profile of $\kappa^{-1}$ is shown in Figure
\ref{fig:ex3_kappa}. Let $\Omega=(0,1)\times(0,1).$ The exact
solution is given by
\[
{\bf u}=\begin{pmatrix}
\sin(2\pi x)\cos(2\pi y)\\
-\cos(2\pi x)\sin(2\pi y)
\end{pmatrix}
\mbox{ and }
p=x^2y^2-1/9.
\]
It is easy to check that $\nabla\cdot{\bf u}=0$ and
$\int_{\Omega}p=0.$ We consider the following permeability
$$
\kappa^{-1}=a(\sin(2\pi x)+1.1),
$$
where $a$ is a given constant. The values of $\kappa^{-1}$ are
plotted in Figure \ref{fig:ex3_kappa} for $a=10^4.$

The optimal convergence rates for the corresponding WG solutions are
presented  in 
Table \ref{ex1_tri_1}-\ref{ex1_L2_3} for $\mu=1,\ 0.01, \mbox{ and }
a=10,\ 10^4$. Our numerical results demonstrate that the WG method
is accurate and robust.

%

\begin{table}[h]
\caption{Example 1. Error and convergence rate for velocity in norm $\3bar Q_h\bu-\bu_h\3bar$ on triangles.
}\label{ex1_tri_1}
\center
\scalebox{0.85}{
\begin{tabular}{||c||cc|cc|cc|cc||}
\hline\hline
&\multicolumn{ 2}{|c}{$a=10,\mu=1$} &\multicolumn{ 2}{|c}{$a=10,\mu=0.01$} &\multicolumn{ 2}{|c}{$a=10^4,\mu=1$} &\multicolumn{ 2}{|c||}{$a=10^4,\mu=0.01$} \\
$h$ & Error &Rate & Error &Rate & Error &Rate & Error &Rate \\
\hline\hline
   1/16 &3.08e-1 &     &1.55e-1 &     &1.58e-1 &     &1.61e-1 &
\\ \hline
   1/24 &2.00e-1 &1.06 &9.90e-2 &1.11 &1.04e-1 &1.02 &1.01e-1 &1.16
\\ \hline
   1/32 &1.49e-1 &1.03 &7.28e-2 &1.07 &7.90e-2 &0.97 &7.37e-2 &1.10
\\ \hline
   1/40 &1.18e-1 &1.02 &5.76e-2 &1.05 &6.33e-2 &0.93 &5.81e-2 &1.06
\\ \hline
   1/48 &9.84e-2 &1.01 &4.77e-2 &1.03 &5.28e-2 &1.00 &4.80e-2 &1.05
\\ \hline
   1/56 &8.43e-2 &1.01 &4.08e-2 &1.02 &4.50e-2 &1.00 &4.09e-2 &1.03
\\ \hline
   1/64 &7.36e-2 &1.01 &3.56e-2 &1.01 &3.94e-2 &1.00 &3.57e-2 &1.03
\\  \hline\hline
\end{tabular} }
\end{table}

\begin{table}[h]
\caption{Example 1. Error and convergence rate for velocity in norm $\|Q_0\bu-\bu_0\|$ on triangles.
}\label{ex1_L2_1}
\center
\scalebox{0.85}{
\begin{tabular}{||c||cc|cc|cc|cc||}
\hline\hline
&\multicolumn{ 2}{|c}{$a=10,\mu=1$} &\multicolumn{ 2}{|c}{$a=10,\mu=0.01$} &\multicolumn{ 2}{|c}{$a=10^4,\mu=1$} &\multicolumn{ 2}{|c||}{$a=10^4,\mu=0.01$} \\
$h$ & Error &Rate & Error &Rate & Error &Rate & Error &Rate \\
\hline\hline
   1/16 &5.01e-2 &     &5.45e-2 &     &5.84e-2 &     &1.77e-2 &
\\ \hline
   1/24 &2.24e-2 &1.98 &2.64e-2 &1.79 &2.58e-2 &2.02 &7.85e-3 &2.01
\\ \hline
   1/32 &1.26e-2 &1.99 &1.53e-2 &1.88 &1.46e-2 &1.97 &4.41e-3 &2.00
\\ \hline
   1/40 &8.09e-3 &2.00 &9.97e-3 &1.93 &9.36e-3 &2.00 &2.82e-3 &2.00
\\ \hline
   1/48 &5.62e-3 &2.00 &6.99e-3 &1.95 &6.49e-3 &2.00 &1.96e-3 &2.00
\\ \hline
   1/56 &4.13e-3 &2.00 &5.16e-3 &1.96 &4.78e-3 &2.00 &1.44e-3 &2.00
\\ \hline
   1/64 &3.16e-3 &2.00 &3.97e-3 &1.97 &3.65e-1 &2.00 &1.10e-3 &2.00
\\  \hline\hline
\end{tabular} }
\end{table}

\begin{table}[h]
\caption{Example 1. Error and convergence rate for velocity in norm $\|\bu-\bu_0\|$ on triangles.
}\label{ex1_L2_2}
\center
\scalebox{0.85}{
\begin{tabular}{||c||cc|cc|cc|cc||}
\hline\hline
&\multicolumn{ 2}{|c}{$a=10,\mu=1$} &\multicolumn{ 2}{|c}{$a=10,\mu=0.01$} &\multicolumn{ 2}{|c}{$a=10^4,\mu=1$} &\multicolumn{ 2}{|c||}{$a=10^4,\mu=0.01$} \\
$h$ & Error &Rate & Error &Rate & Error &Rate & Error &Rate \\
\hline\hline
   1/16 &3.12e-2 &     &6.70e-2 &     &5.27e-2 &     &6.48e-3 &
\\ \hline
   1/24 &1.39e-2 &1.99 &3.22e-2 &1.80 &2.33e-2 &2.01 &2.98e-3 &1.92
\\ \hline
   1/32 &7.86e-3 &1.99 &1.87e-2 &1.89 &1.30e-2 &2.01 &1.69e-3 &1.96
\\ \hline
   1/40 &5.04e-3 &2.00 &1.21e-2 &1.93 &8.44e-3 &1.97 &1.09e-3 &1.98
\\ \hline
   1/48 &3.50e-3 &2.00 &8.50e-3 &1.95 &5.86e-3 &2.00 &7.56e-4 &1.99
\\ \hline
   1/56 &2.57e-3 &2.00 &6.28e-3 &1.97 &4.30e-3 &2.00 &5.56e-4 &2.00
\\ \hline
   1/64 &1.97e-3 &2.00 &4.82e-3 &1.98 &3.29e-3 &2.00 &4.26e-4 &2.00
\\  \hline\hline
\end{tabular} }
\end{table}

\begin{table}[h]
\caption{Example 1. Error and convergence rate for pressure in norm $\|\mathbb{Q}_hp-p_h\|$ on triangles.
}\label{ex1_L2_3}
\center
\scalebox{0.85}{
\begin{tabular}{||c||cc|cc|cc|cc||}
\hline\hline
&\multicolumn{ 2}{|c}{$a=10,\mu=1$} &\multicolumn{ 2}{|c}{$a=10,\mu=0.01$} &\multicolumn{ 2}{|c}{$a=10^4,\mu=1$} &\multicolumn{ 2}{|c||}{$a=10^4,\mu=0.01$} \\
$h$ & Error &Rate & Error &Rate & Error &Rate & Error &Rate \\
\hline\hline
   1/16 &1.17e-1 &     &4.60e-2 &     &4.97e-1 &     &5.46e-2 &
\\ \hline
   1/24 &7.81e-2 &1.01 &3.24e-2 &0.86 &3.30e-1 &1.00 &3.57e-2 &1.05
\\ \hline
   1/32 &5.85e-2 &1.00 &2.49e-2 &0.92 &2.47e-1 &1.00 &2.64e-2 &1.05
\\ \hline
   1/40 &4.68e-2 &1.00 &2.01e-2 &0.95 &1.98e-1 &1.00 &2.09e-2 &1.04
\\ \hline
   1/48 &3.90e-2 &1.00 &1.69e-2 &0.97 &1.66e-1 &0.97 &1.74e-2 &1.03
\\ \hline
   1/56 &3.34e-2 &1.00 &1.45e-2 &0.98 &1.42e-1 &1.00 &1.48e-2 &1.02
\\ \hline
   1/64 &2.92e-2 &1.00 &1.27e-2 &0.98 &1.24e-1 &1.00 &1.29e-2 &1.02
\\  \hline\hline
\end{tabular} }
\end{table}

The rest of the test problems have the following data setting:
\begin{equation}\label{setting}
\Omega=(0,1)\times(0,1), \quad\mu=0.01,\quad {\bf f}=0,\quad
{\bf g}=\begin{pmatrix}
1\\
0
\end{pmatrix}.
\end{equation}

\subsection{Example 2}
\begin{figure}[H]
\centering
\begin{tabular}{cc}
  \resizebox{2.2in}{2.2in}{\includegraphics{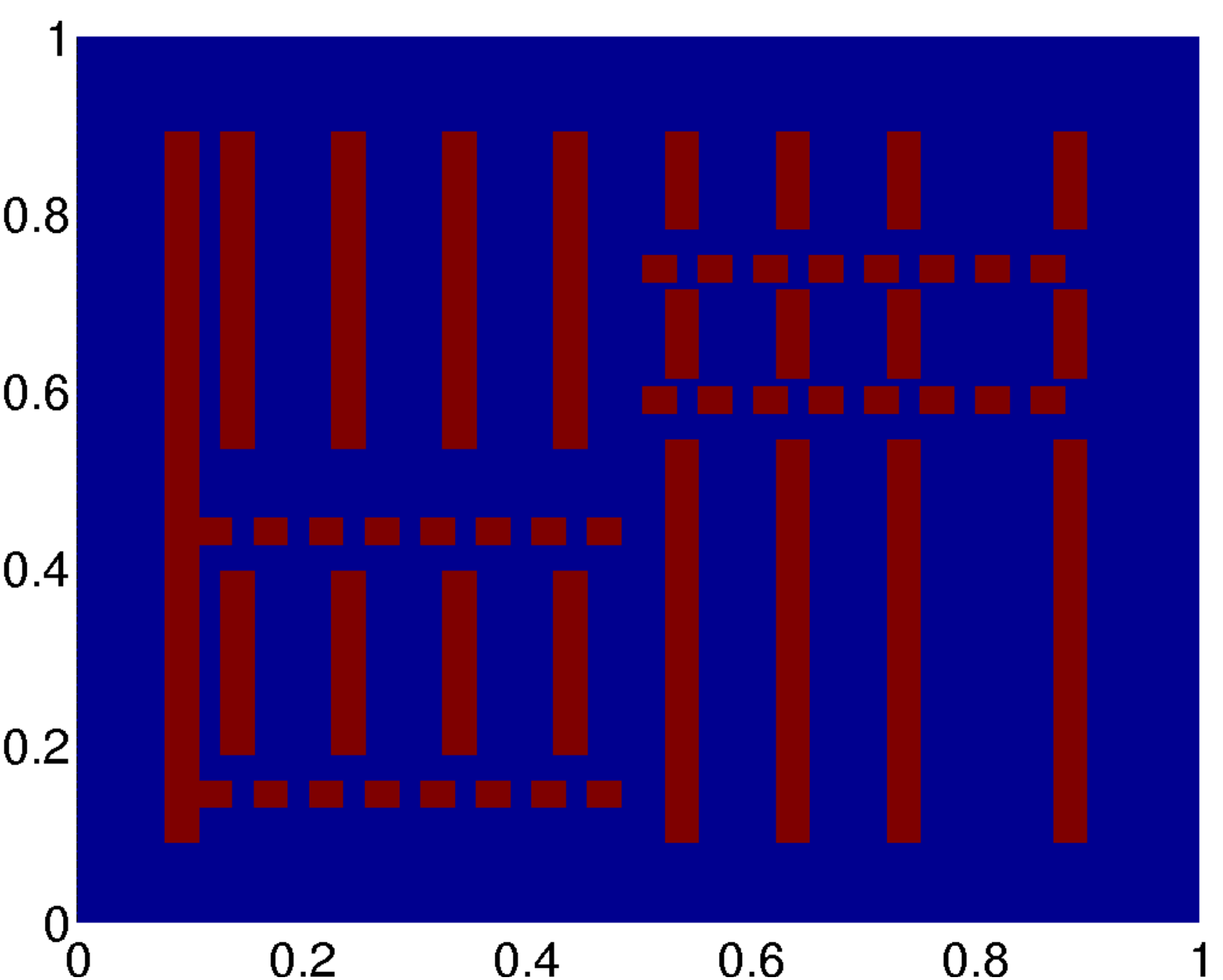}}&
  \resizebox{2.2in}{2.2in}{\includegraphics{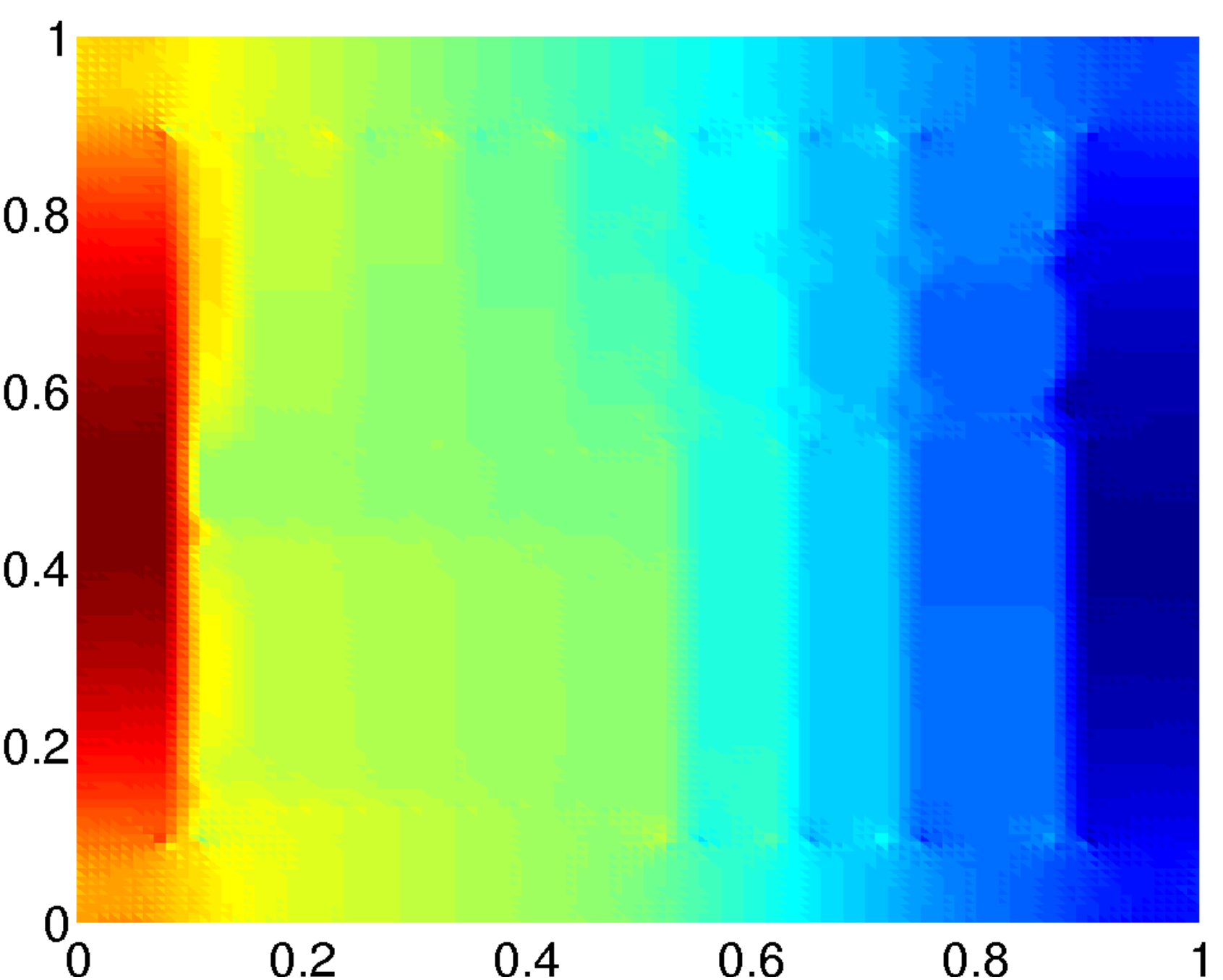}}\\
  (a) & (b)
\end{tabular}
\caption{Example 2: (a) Profile of $\kappa^{-1}$ with low (blue) and
high (red); (b) Pressure profile.}\label{fig:ex4_1}
\end{figure}

\begin{figure}[H]
\centering
\begin{tabular}{cc}
  \resizebox{2.2in}{2.2in}{\includegraphics{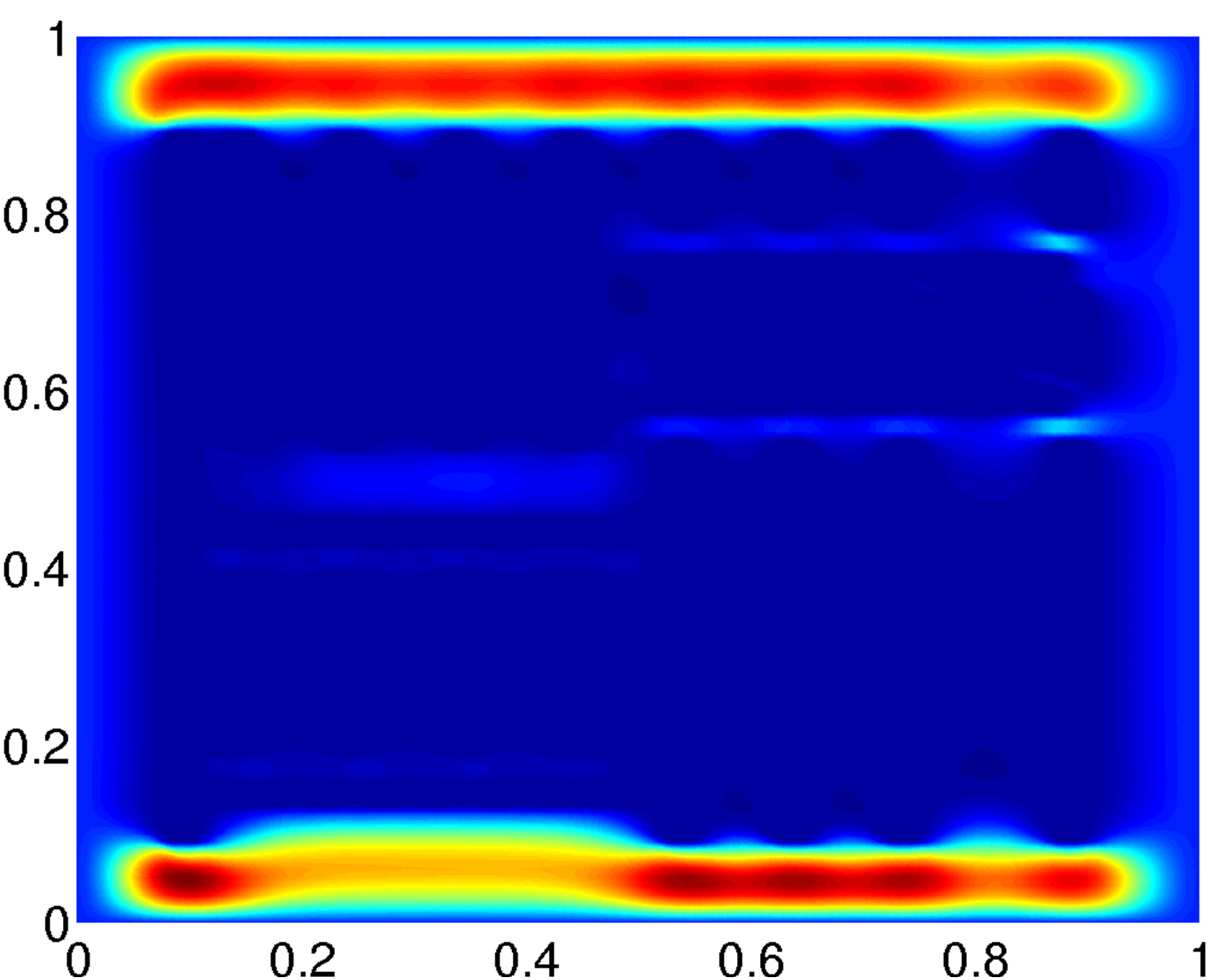}}
  &\resizebox{2.2in}{2.2in}{\includegraphics{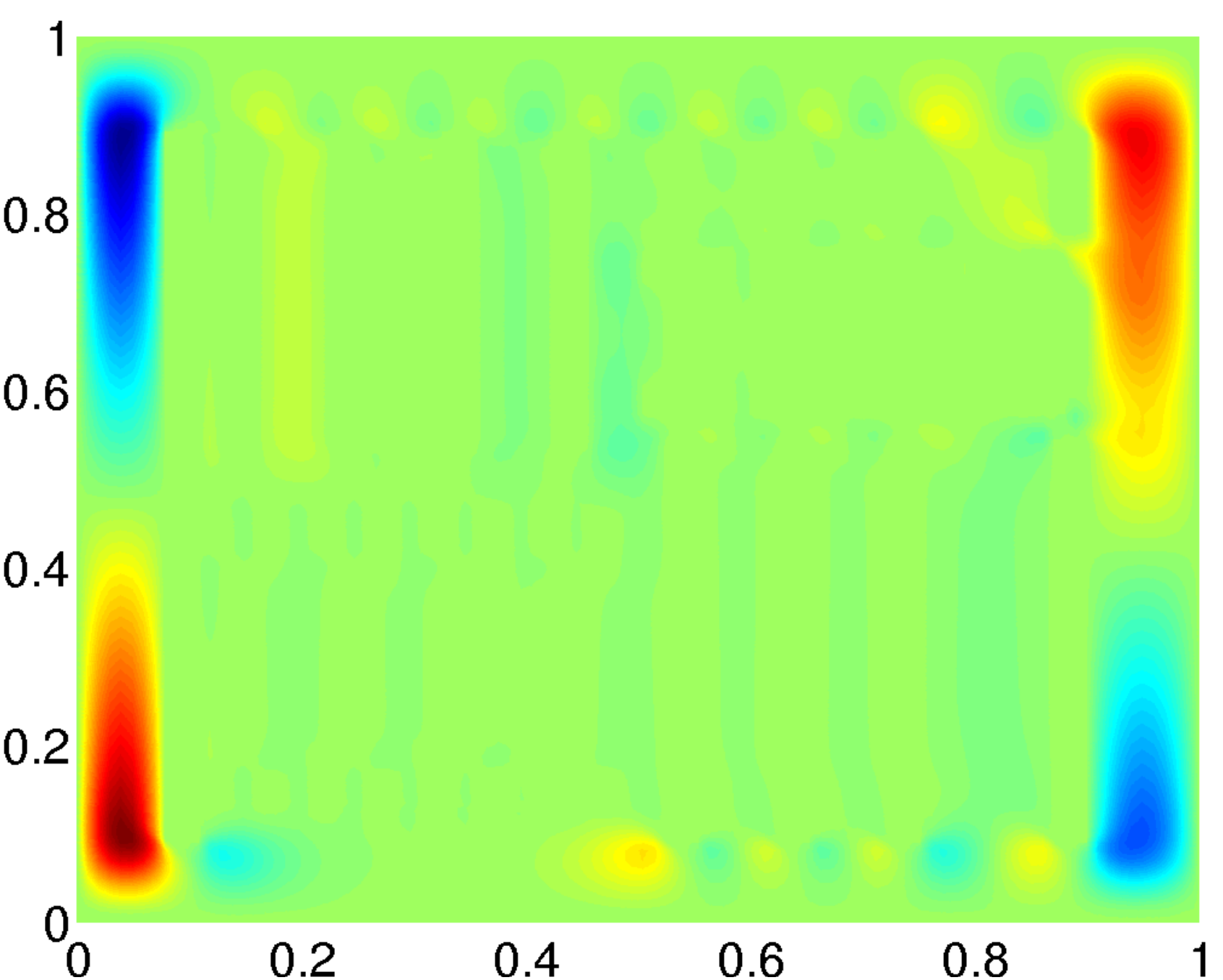}}\\
  (a) & (b)
\end{tabular}
\caption{Example 2: (a) First component of velocity $u_1$; (b)
Second component of velocity  $u_2$.}\label{fig:ex4_2}
\end{figure}

The Brinkman equations (\ref{moment})-(\ref{bc}) are solved over a
region with a high contrast permeability. The profile of the
permeability inverse is plotted in Figure \ref{fig:ex4_1} (a) with
$\kappa^{-1}_{\mbox{min}}=1$ and $\kappa_{\mbox{max}}=10^6$ in the
red and blue regions.

\vskip.1in

A $100\times 100$ mesh is used for plotting Figure \ref{fig:ex4_1}
and Figure \ref{fig:ex4_2}. The pressure profile of the WG method is
presented in Figure \ref{fig:ex4_1} (b). The first and the second
components of the velocity calculated by the WG method are shown in
Figure \ref{fig:ex4_2}(a) and (b) respectively.

\subsection{Example 3}

\begin{figure}[H]
\centering
\begin{tabular}{cc}
  \resizebox{2.2in}{2.2in}{\includegraphics{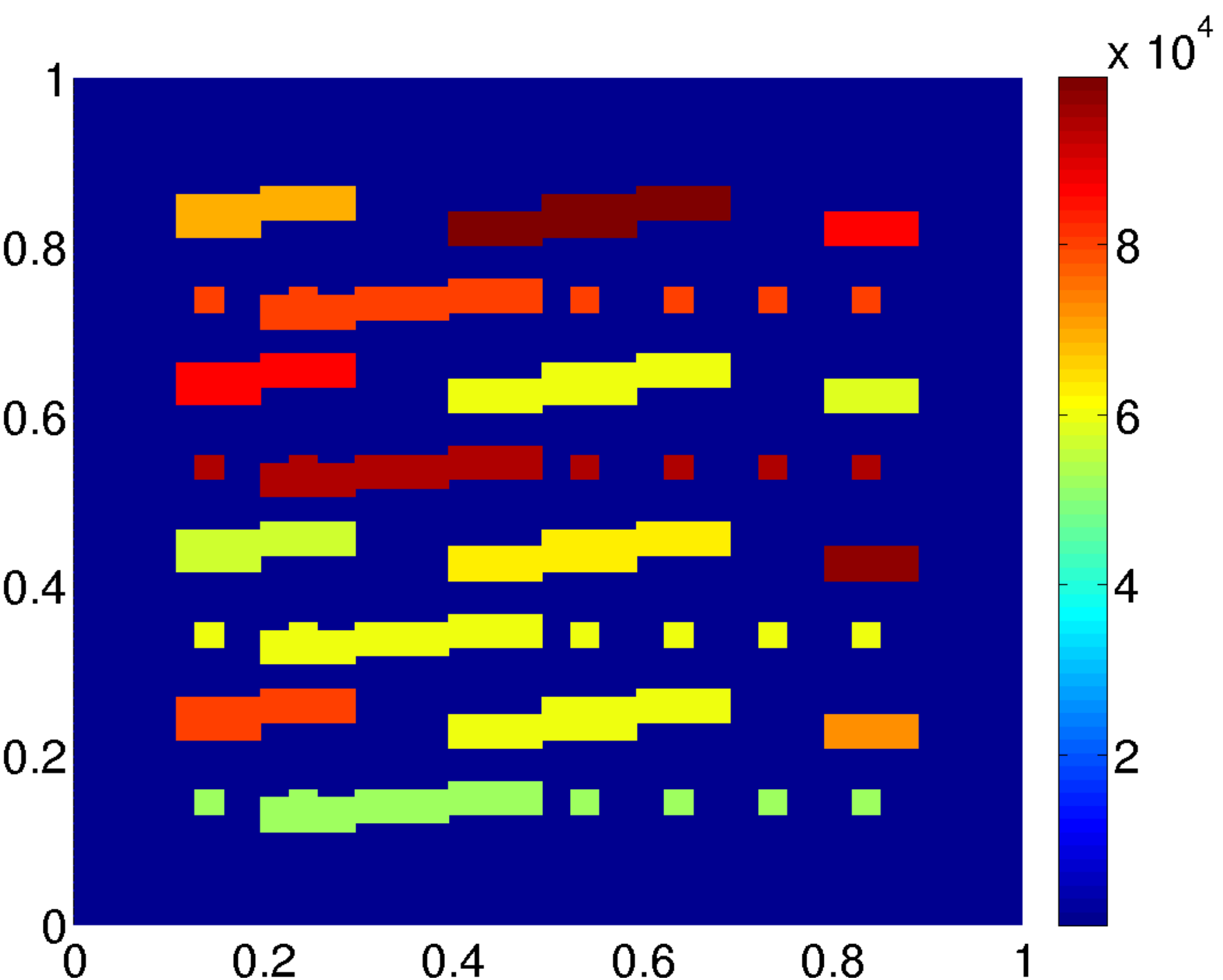}}&
  \resizebox{2.2in}{2.2in}{\includegraphics{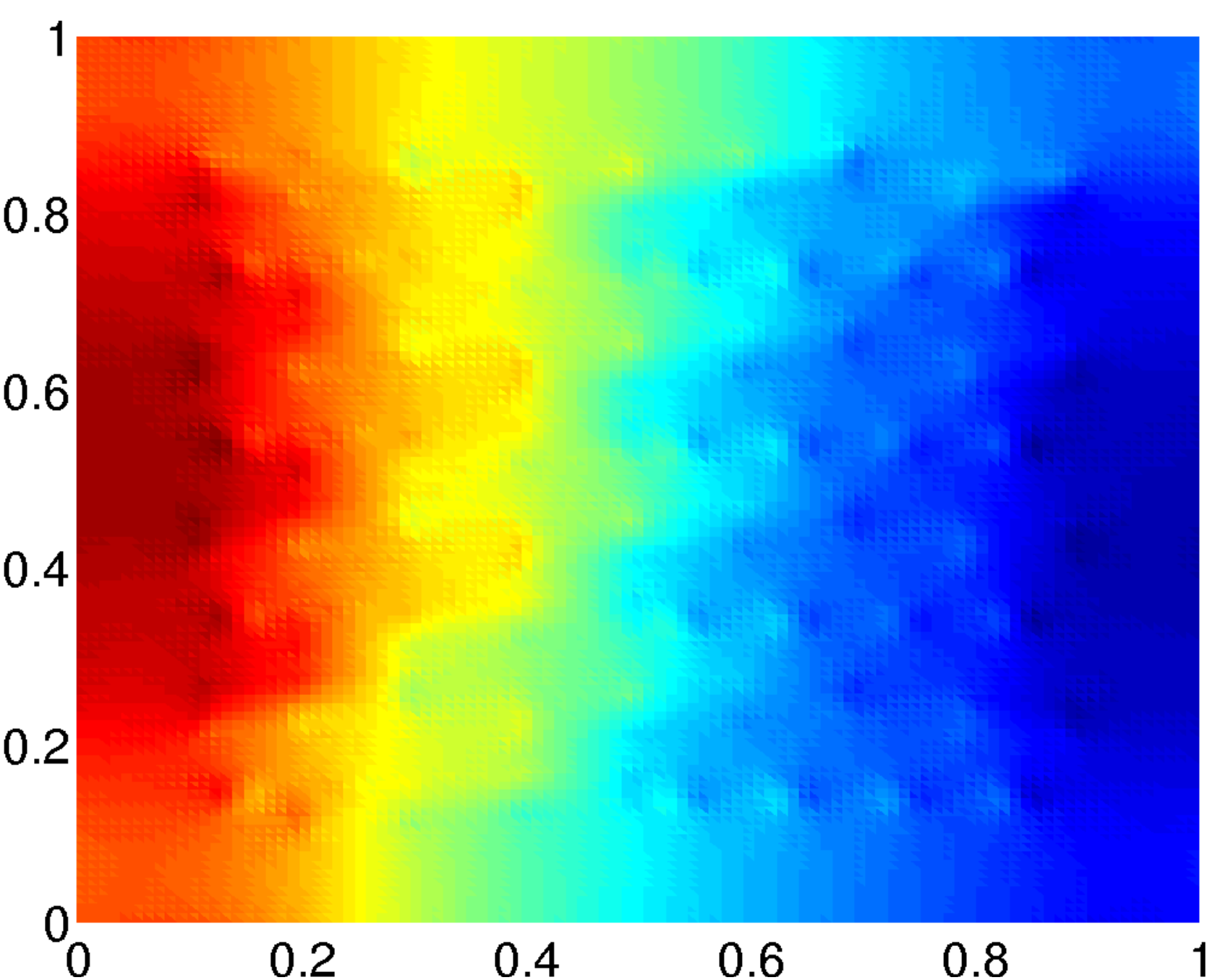}}\\
  (a) &(b)
\end{tabular}
\caption{Example 3: (a) Profile of $\kappa^{-1}$; (b) Pressure
profile.} \label{fig:ex5_1}
\end{figure}

\begin{figure}[H]
\centering
\begin{tabular}{cc}
  \resizebox{2.2in}{2.2in}{\includegraphics{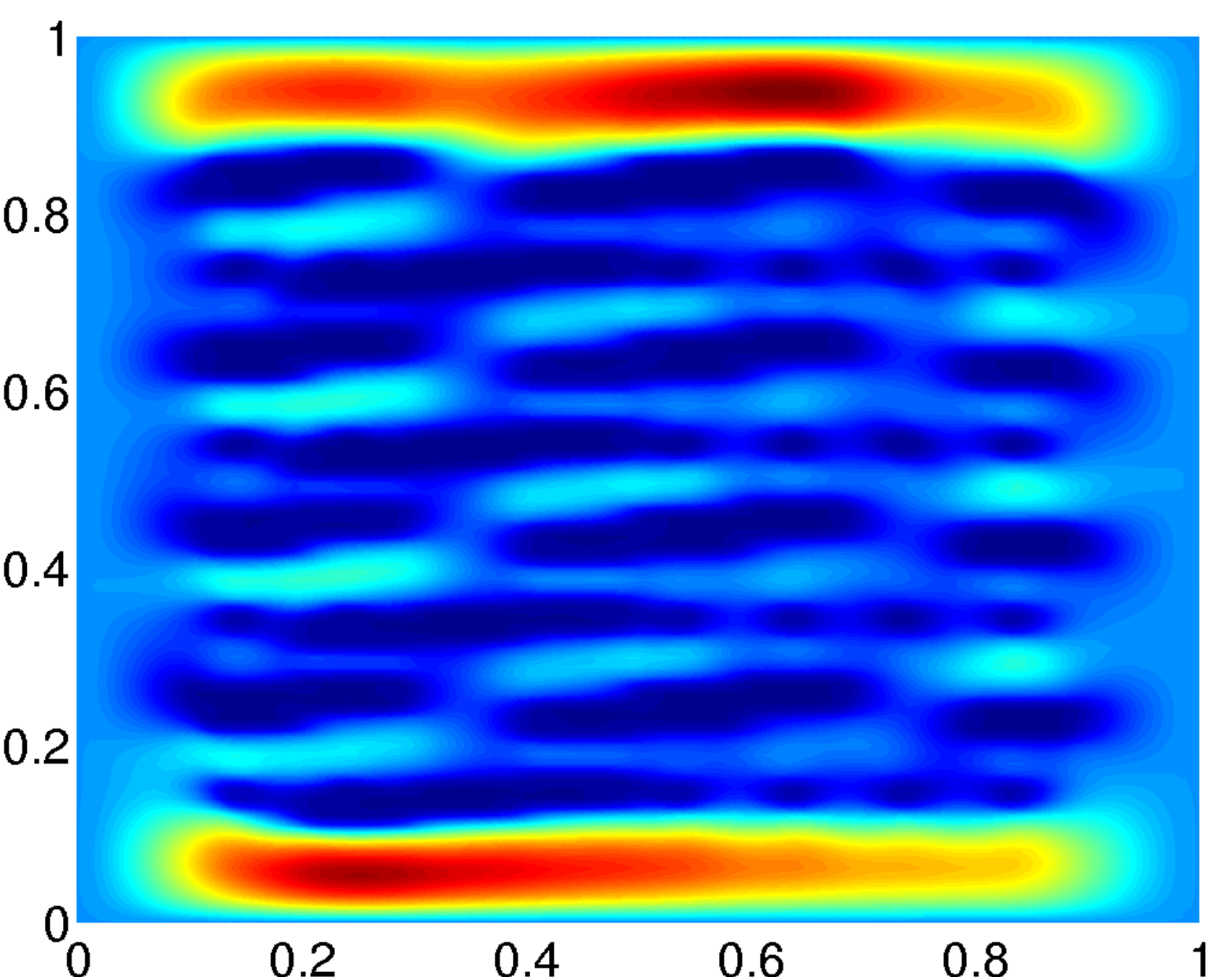}}&
  \resizebox{2.2in}{2.2in}{\includegraphics{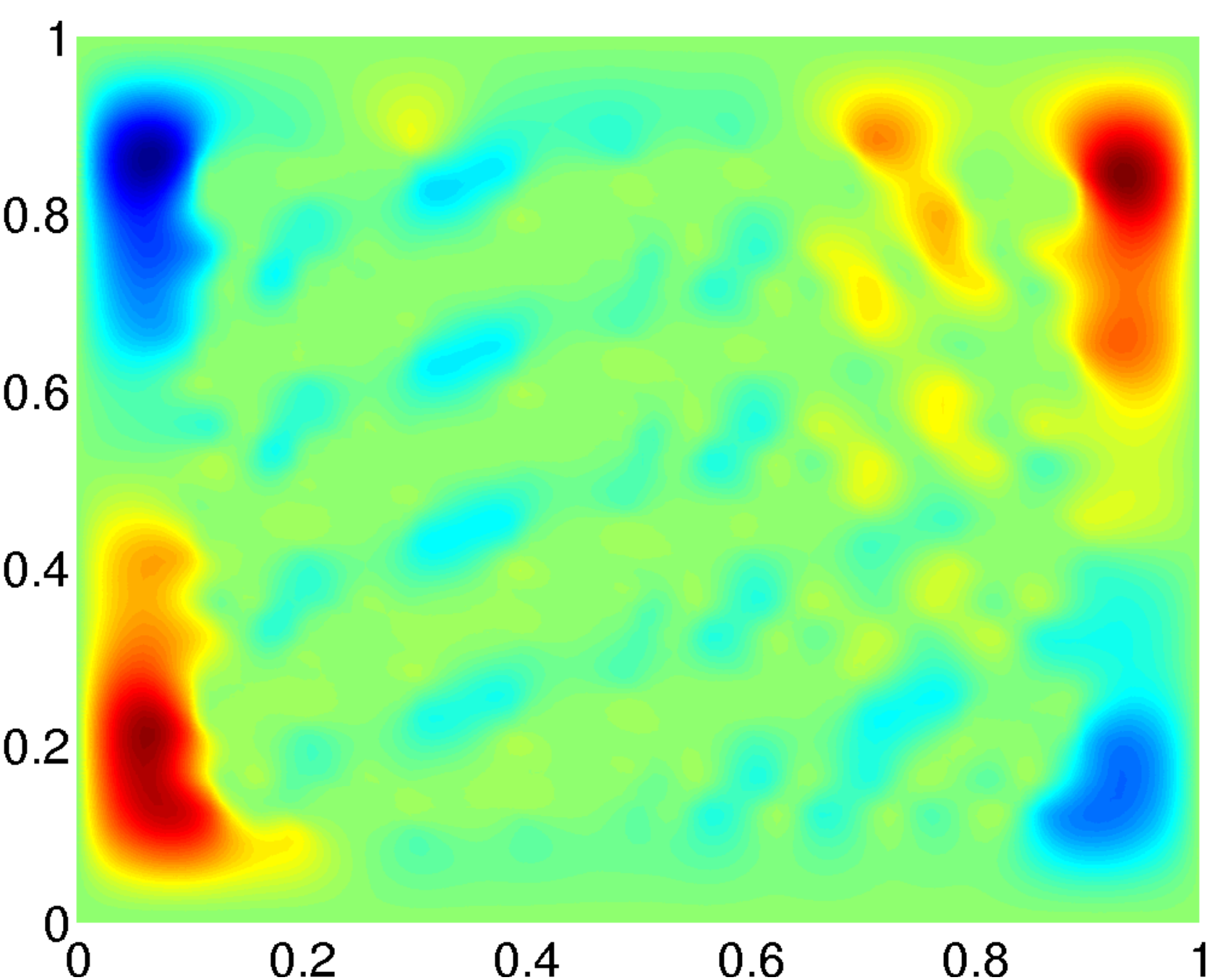}}\\
  (a) & (b)
\end{tabular}
\caption{Example 3: (a) First component of velocity $u_1$; (b)
Second component of velocity  $u_2$.} \label{fig:ex5_2}
\end{figure}

This is another example of flow through a region with high contrast
permeability. The profile of $\kappa^{-1}$ is plotted in Figure
\ref{fig:ex5_1}(a) and the data for the modeling equation is given
in (\ref{setting}).

A $100\times 100$ mesh is used for plotting Figure \ref{fig:ex5_1}
and Figure \ref{fig:ex5_2}. The pressure profile of the WG method is
presented in Figure \ref{fig:ex5_1}(b). The first and the second
components of the velocity calculated by the WG method are shown in
Figure \ref{fig:ex5_2}(a) and \ref{fig:ex5_2}(b) respectively.

\medskip

The rest of the examples simulate flow through porous media with
different geometries without known analytical solutions.  Flow
through vuggy media, fibrous materials and open foam geometries are
tested and their permeability inverse profiles can be found in
different literatures such as
\cite{iliev2011variational,willems2009numerical}.

\subsection{Example 4}
\begin{figure}[H]
\centering
\begin{tabular}{cc}
  \resizebox{2.2in}{2.2in}{\includegraphics{ex3_kappa_P1P1.pdf}}&
  \resizebox{2.2in}{2.2in}{\includegraphics{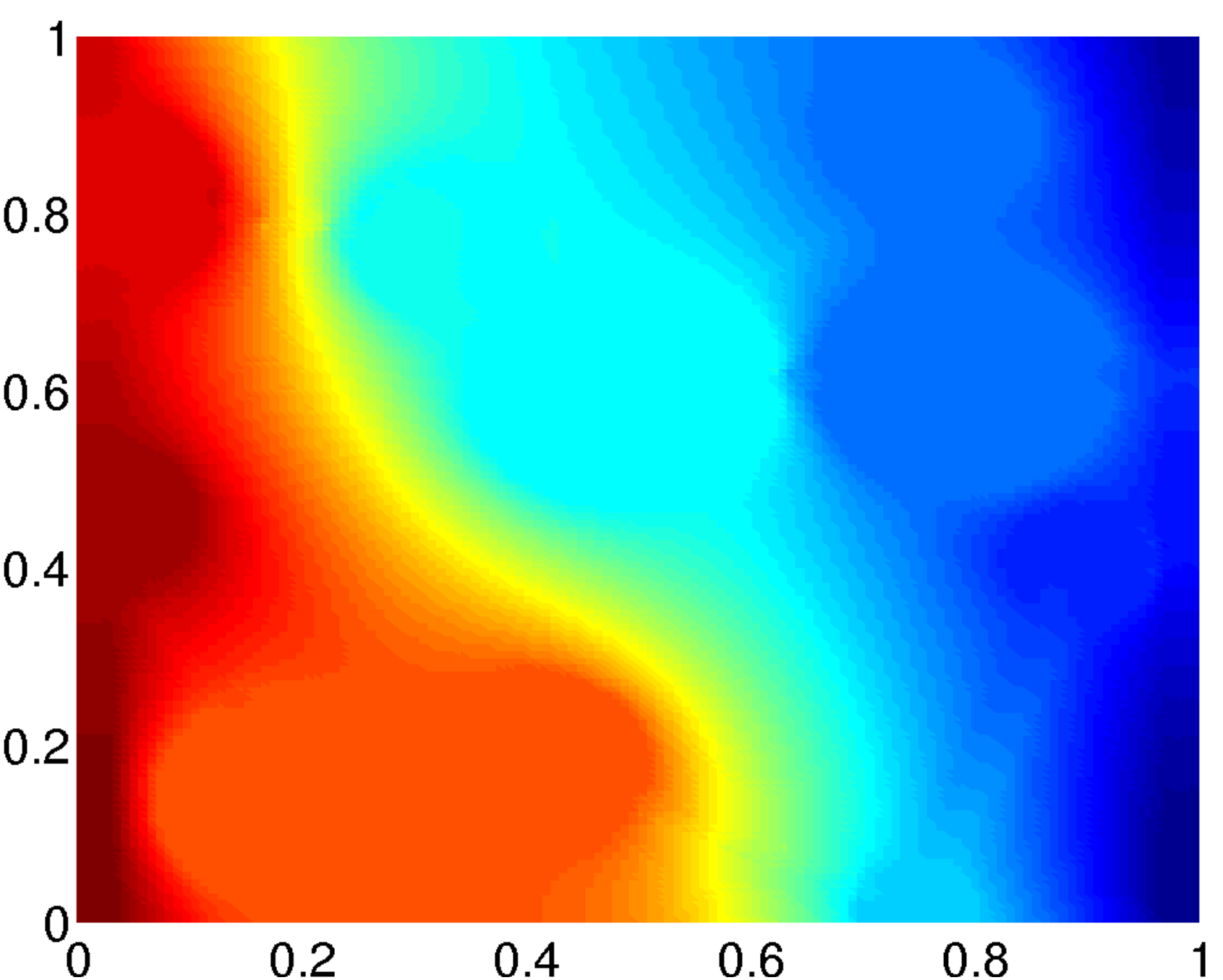}}
  \\
  (a) & (b)
\end{tabular}
\caption{Example 4: (a) Profile of $\kappa^{-1}$ for vuggy medium;
(b) Pressure profile.}\label{fig:ex6_1}
\end{figure}

\begin{figure}[H]
\centering
\begin{tabular}{cc}
  \resizebox{2.2in}{2.2in}{\includegraphics{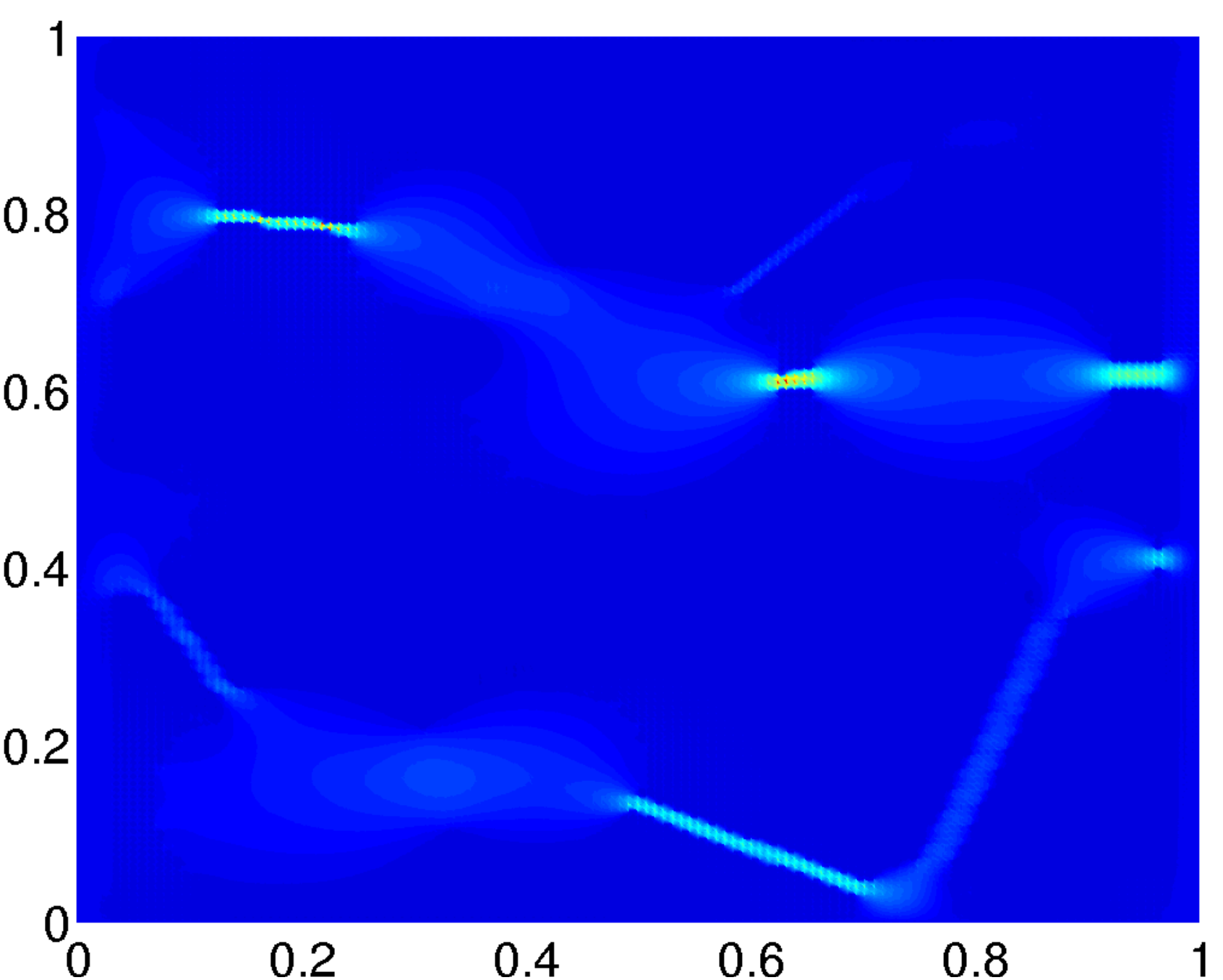}}&
  \resizebox{2.2in}{2.2in}{\includegraphics{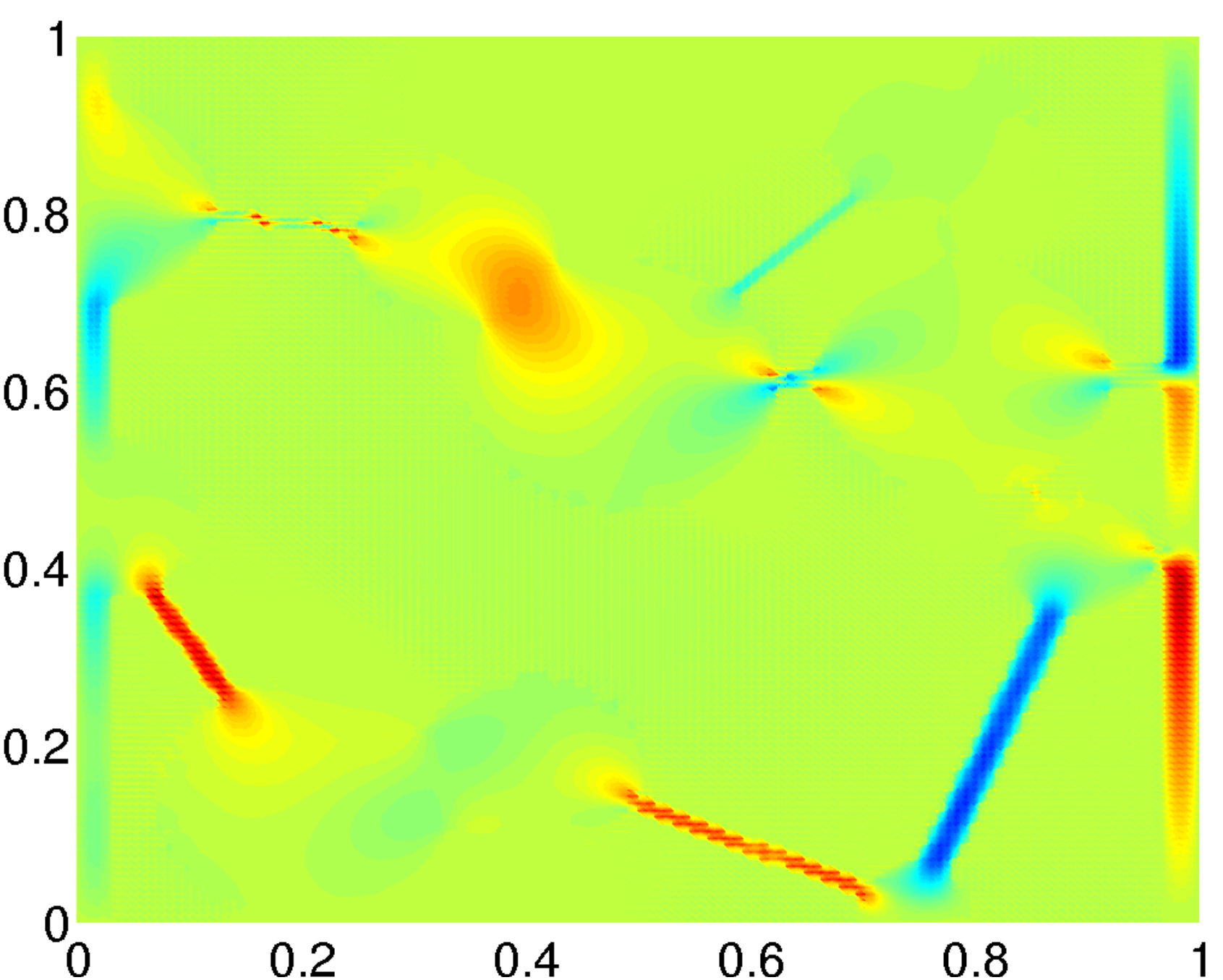}}\\
  (a) & (b)
\end{tabular}
\caption{Example 4: (a) First component of velocity $u_1$; (b)
Second component of velocity  $u_2$.}\label{fig:ex6_2}
\end{figure}

In this example, the Brinkman equations (\ref{moment})-(\ref{bc})
are solved over a vuggy medium with the data set in (\ref{setting}).
The profile of $\kappa^{-1}$ is plotted in Figure
\ref{fig:ex6_1}(a).

For this example, a $128\times 128$ mesh is used for plotting Figure
\ref{fig:ex6_1} and Figure \ref{fig:ex6_2}. The pressure profile of
the WG method is presented in Figure \ref{fig:ex6_1}(b). The first
and the second components of the velocity calculated by the WG
method are shown in Figure \ref{fig:ex6_2}(a) and \ref{fig:ex6_2}(b)
respectively.

\subsection{Example 5}

\begin{figure}[H]
\centering
\begin{tabular}{cc}
  \resizebox{2.2in}{2.2in}{\includegraphics{ex_fibrous_kappa_P1P1.pdf}}&
  \resizebox{2.2in}{2.2in}{\includegraphics{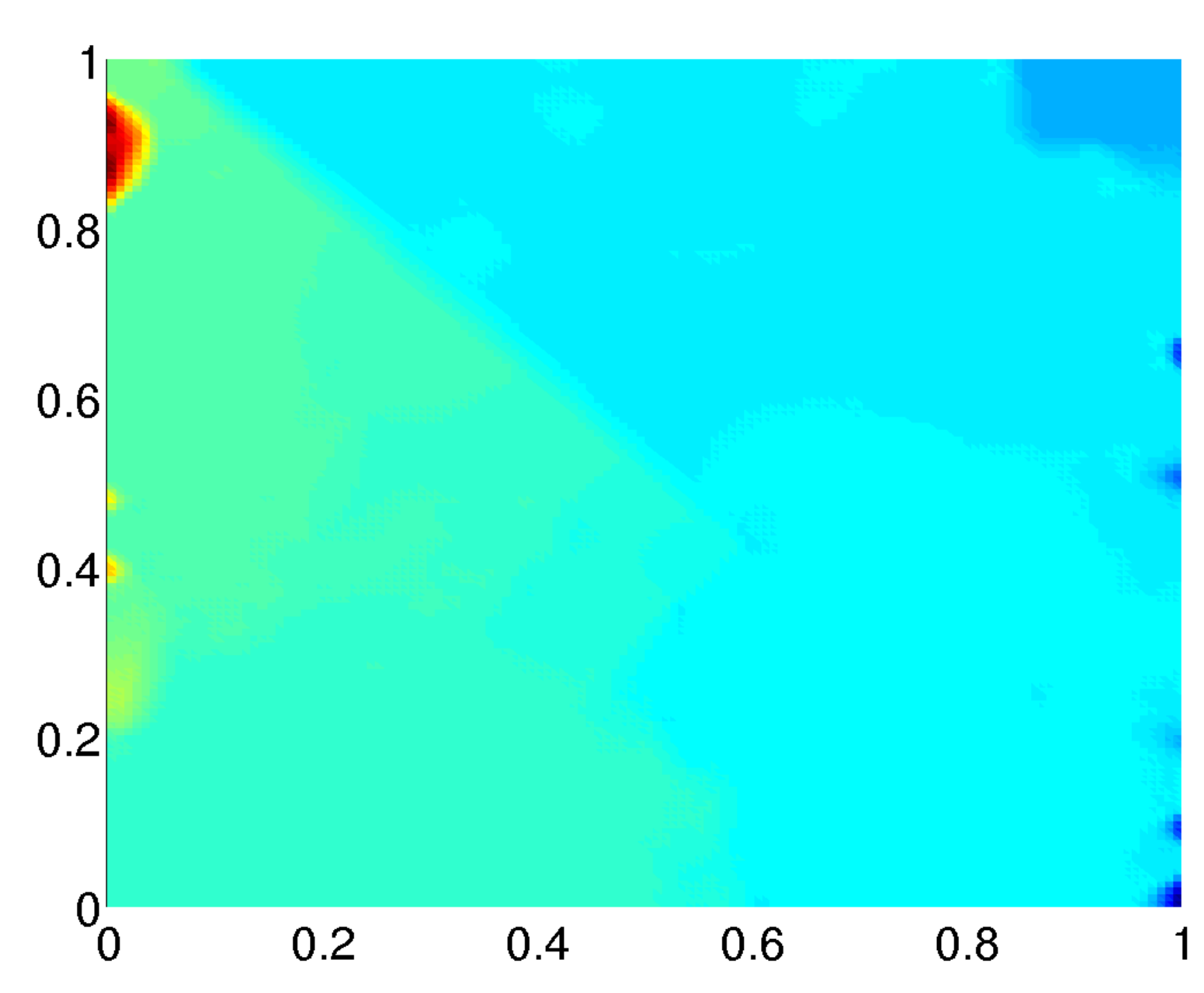}}\\
  (a) & (b)
\end{tabular}
\caption{Example 5: (a) Profile of $\kappa^{-1}$ for fibrous
structure; (b) Pressure profile.}\label{fig:ex7_1}
\end{figure}

\begin{figure}[H]
\centering
\begin{tabular}{cc}
  \resizebox{2.2in}{2.2in}{\includegraphics{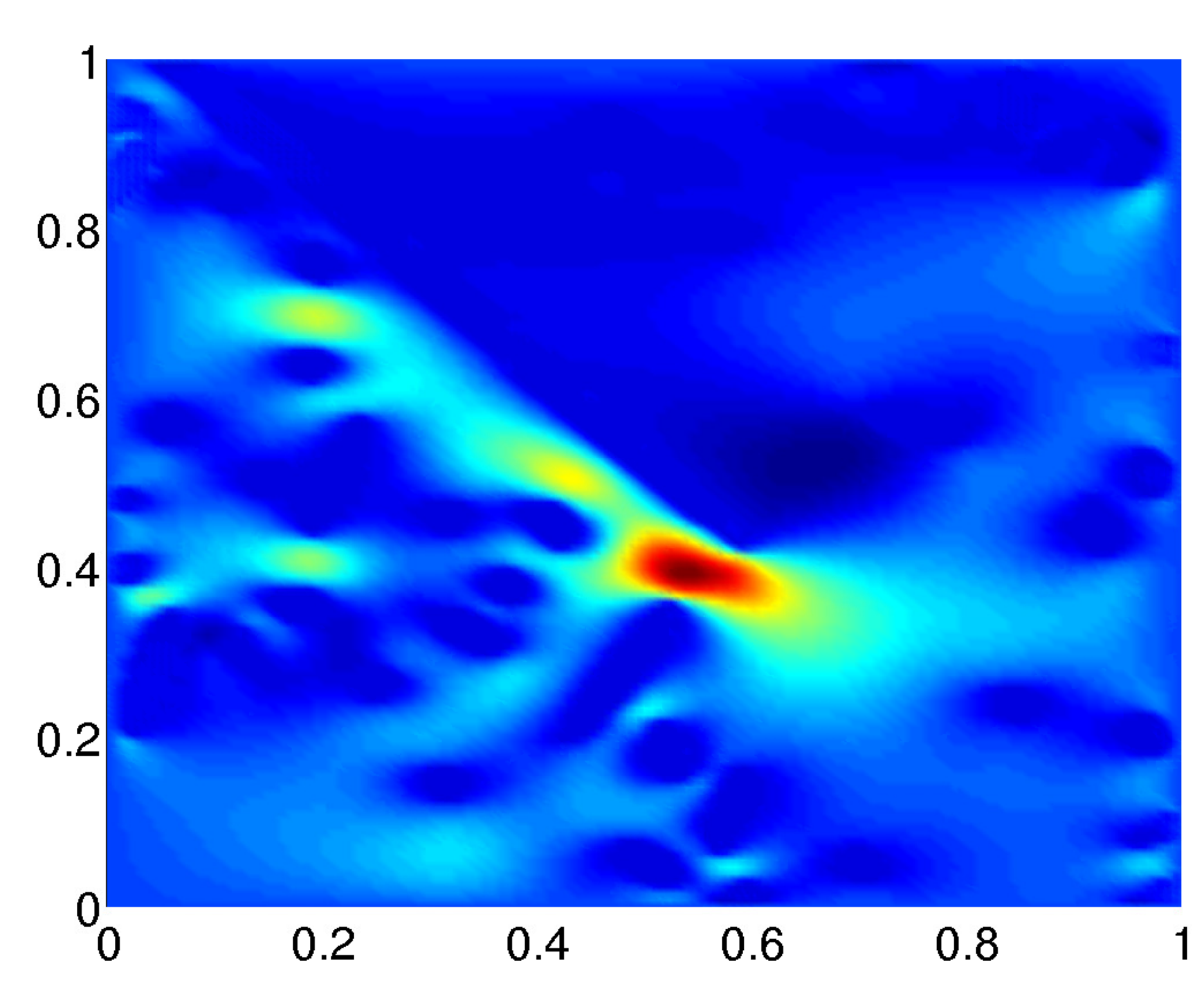}}&
  \resizebox{2.2in}{2.2in}{\includegraphics{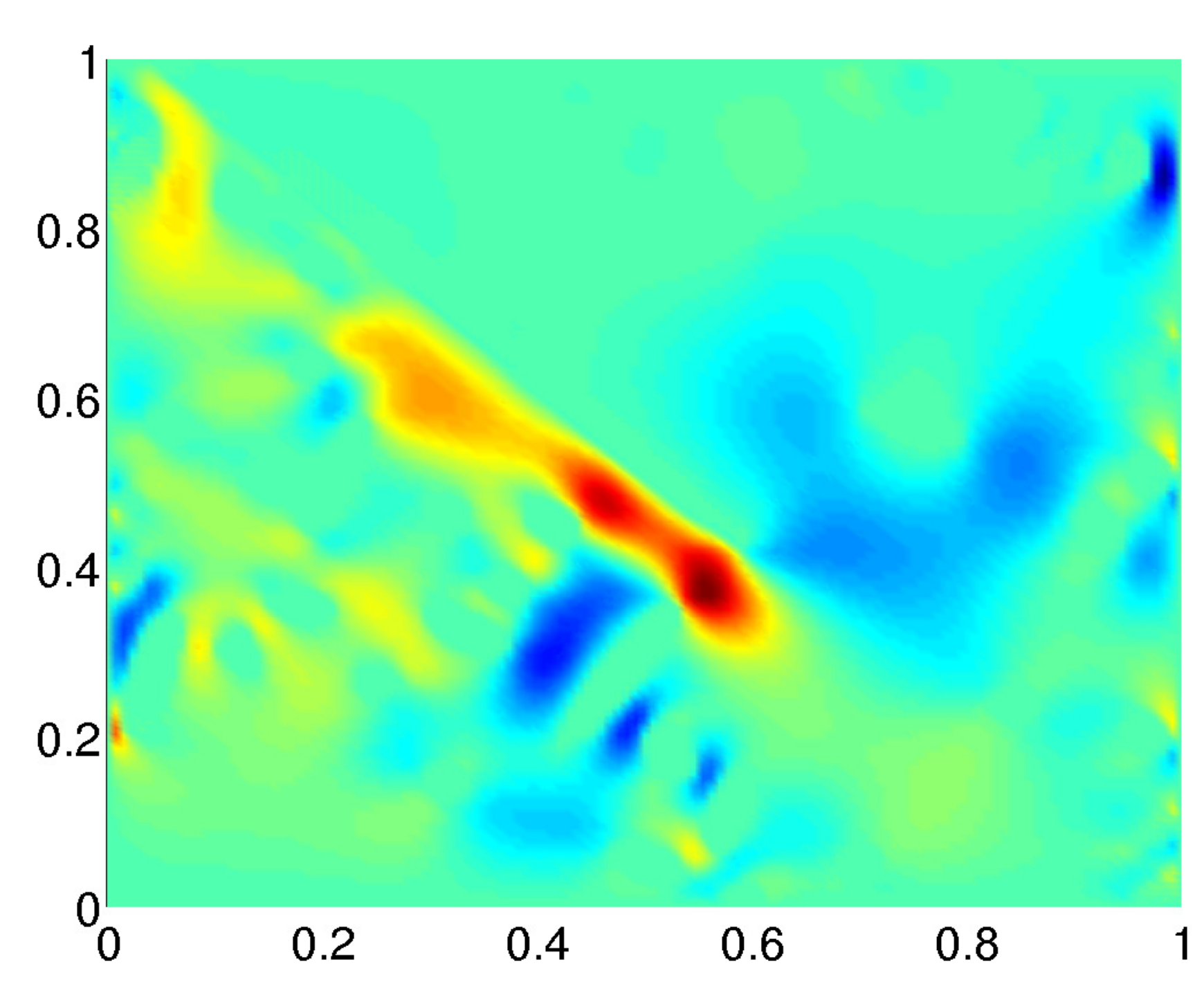}}\\
  (a) & (b)
\end{tabular}
\caption{Example 5: (a) First component of velocity $u_1$; (b)
Second component of velocity  $u_2$.}\label{fig:ex7_2}
\end{figure}

This example is frequently used in filtration and insulation
materials. The inverse of permeability of fibrous structure is shown
in Figure \ref{fig:ex7_1}(a).  A $128\times 128$ mesh is used for
plotting Figure \ref{fig:ex7_1} and Figure \ref{fig:ex7_2}. The
pressure profile of the WG method is presented in Figure
\ref{fig:ex7_1}(b). The first and the second components of the
velocity calculated by the WG method are shown in Figure
\ref{fig:ex7_2}(a) and \ref{fig:ex7_2}(b) respectively.

\subsection{Example 6}

\begin{figure}[H]
\centering
\begin{tabular}{cc}
  \resizebox{2.2in}{2.2in}{\includegraphics{ex5_kappa_P1P1.pdf}}&
  \resizebox{2.2in}{2.2in}{\includegraphics{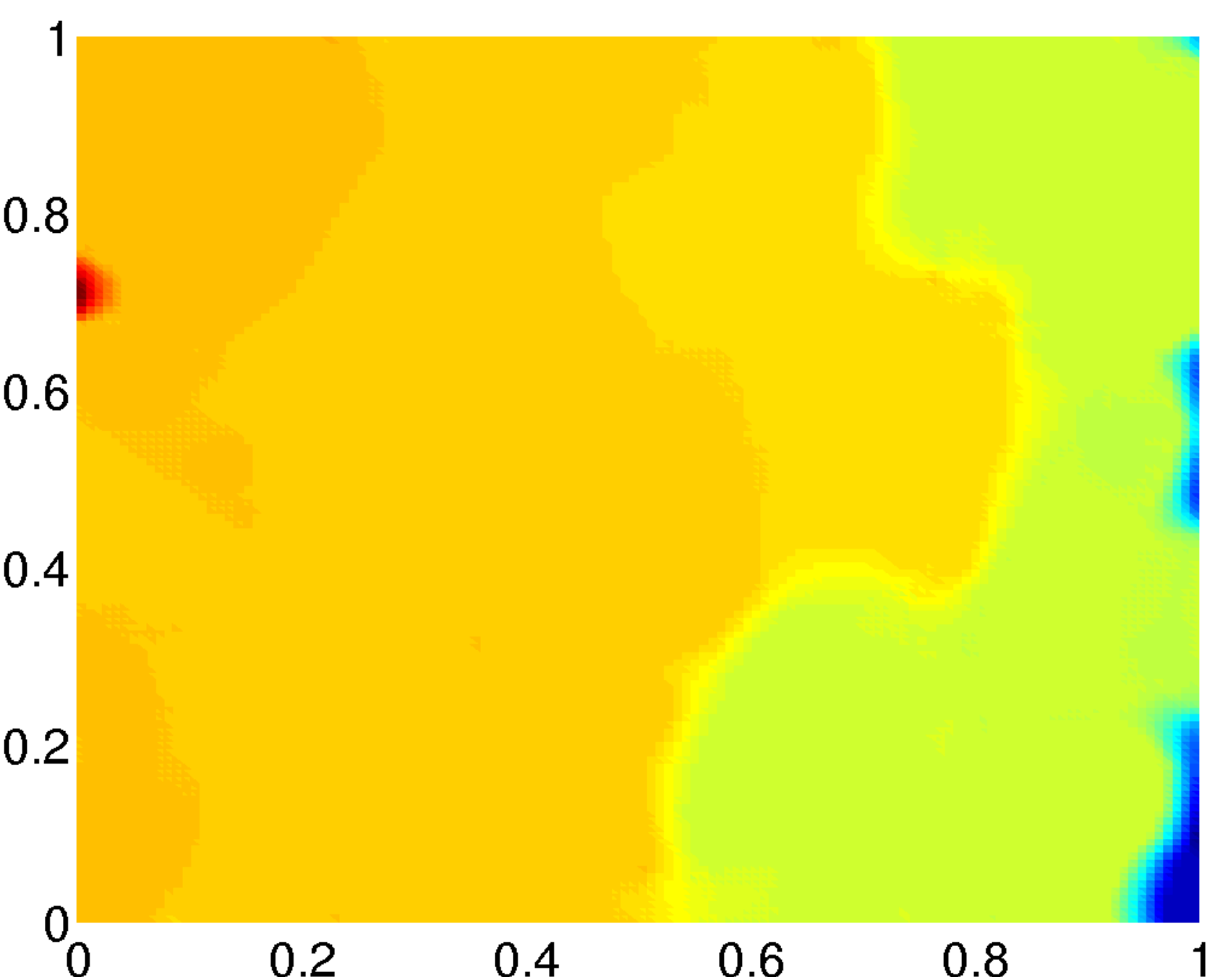}}\\
  (a)&(b)
\end{tabular}
\caption{Example 6: (a) Profile of $\kappa^{-1}$ for open form; (b)
Pressure profile.}\label{fig:ex8_1}
\end{figure}

\begin{figure}[H]
\centering
\begin{tabular}{cc}
  \resizebox{2.2in}{2.2in}{\includegraphics{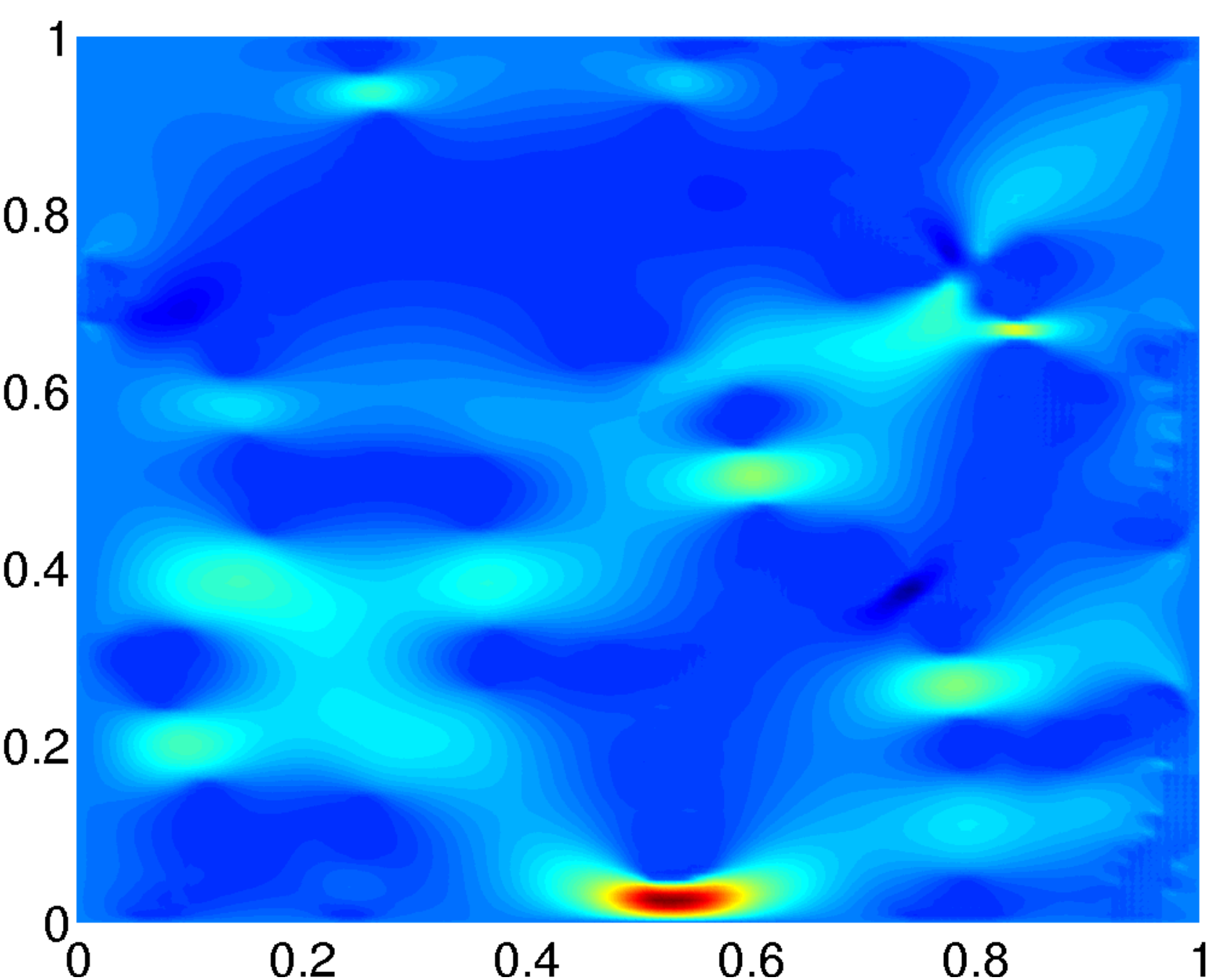}}&
  \resizebox{2.2in}{2.2in}{\includegraphics{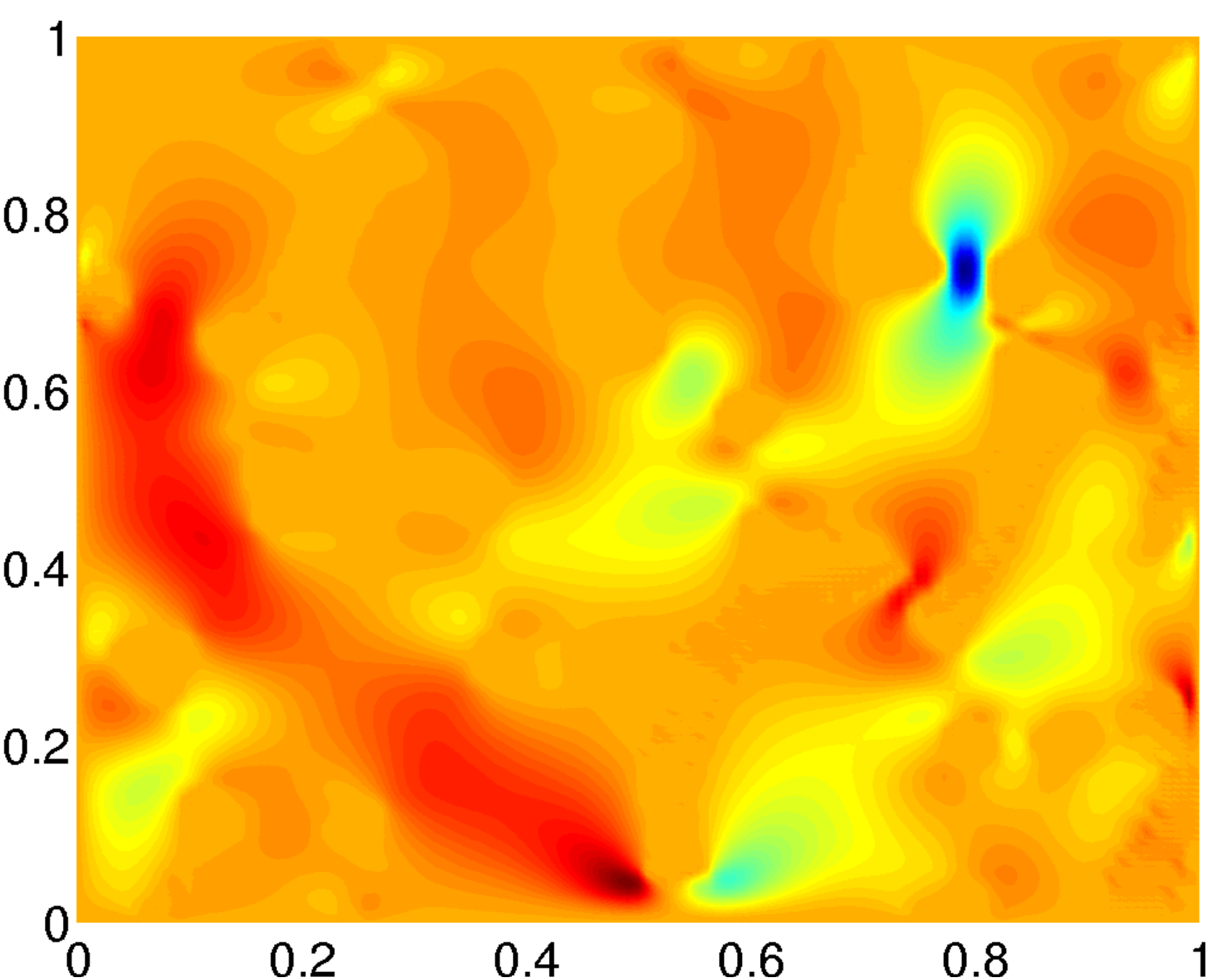}}\\
  (a)&(b)
\end{tabular}
\caption{Example 6: (a) First component of velocity $u_1$; (b)
Second component of velocity  $u_2$.}\label{fig:ex8_2}
\end{figure}

The geometry of this example is an open foam with a profile of
$\kappa^{-1}$ shown in Figure \ref{fig:ex8_1}(a). The rest of the
modeling data is given in (\ref{setting}). Figure \ref{fig:ex8_1}
and Figure \ref{fig:ex8_2} are plotted over a $128\times 128$ grid.
The profiles of the approximate pressure and velocity are presented
in Figure \ref{fig:ex8_1}(b) and  Figure \ref{fig:ex8_2}
respectively.


\begin{thebibliography}{99}

\bibitem{bc}
{\sc S. Badia and R. Codina}
{\em  Unified stabilized finite element formulations for
the Stokes and the Darcy problems}, SIAM J. Numer. Anal., 47 (2009), 1971-2000.

\bibitem{babuska}
{\sc I. Babu\u{s}ka},
{\em The finite element method with Lagrangian multiplier},
Numer. Math., 20 (1973), 179--192.

\bibitem{bs}
{\sc S. Brenner and R. Scott},
{\em Mathematical theory of finite element methods}, Springer, 2002.

\bibitem{brezzi}
{\sc F. Brezzi},
{\em On the existence, uniqueness, and approximation of saddle point
problems arising from Lagrangian multipliers},
RAIRO, Anal. Num\'{e}r.,  2 (1974), 129--151.

\bibitem{bf}
{\sc F. Brezzi and M. Fortin},
{\em Mixed and Hybrid Finite Elements},
Springer-Verlag, New York, 1991.

\bibitem{cr}
{\sc M. Crouzeix and P.  Raviart},  {\em Conforming and nonconforming
finite element methods for solving the stationary Stokes
equations}, RAIRO Anal. Numer., 7 (1973), 33--76.

\bibitem{gr}
{\sc V. Girault and P. Raviart},
{\em Finite Element Methods for
the Navier-Stokes Equations: Theory and Algorithms}, Springer-Verlag, Berlin, 1986.

\bibitem{gun}
{\sc M.  Gunzburger},
{\em Finite Element Methods for Viscous Incompressible Flows,
A Guide to Theory, Practice and Algorithms}, Academic,
San Diego, 1989.

\bibitem{efendiev2011robust}
{\sc Y. Efendiev, J. Galvis, R. Lazarov, and J. Willems},
{\em  Robust domain decomposition preconditioners for abstract symmetric
  positive definite bilinear forms}, arXiv:1105.1131.

\bibitem{iliev2011variational}
{\sc O. Iliev, R. Lazarov, and J. Willems},
{\em  Variational multiscale finite element method for flows in highly
  porous media}, Multiscale Modeling \& Simulation, 9(4) (2011), 1350--1372.

\bibitem{ks}
{\sc J. Könnö and R. Stenberg}, {\em H(div)-conforming finite
elements for the Brinkman problem}, Math. Models and Meth. Applied
Sciences, 11 (2011),  2227--2248.

\bibitem{mtw}
{\sc K. Mardal, X. Tai, and R. Winther} {\em A Robust finite element
method for Darcy-Stokes flow}, SIAM J. Numer. Anal., 40 (2002),
1605–-1631.

\bibitem{wy}
{\sc J. Wang and X. Ye}, {\em A weak Galerkin finite element method
for second-order elliptic problems}, J. Comp. and Appl. Math, 241
(2013) 103-115, arXiv:1104.2897v1.

\bibitem{wy-mixed}
{\sc J. Wang and X. Ye}, {\em A Weak Galerkin mixed finite element
method for second-order elliptic problems}, Math. Comp., to appear,
arXiv:1202.3655v1.

\bibitem{wg-stokes}
{\sc J. Wang and X. Ye}, {\em A Weak Galerkin finite element method
for the Stokes equations}, arXiv:1302.2707v1.

\bibitem{willems2009numerical}
{\sc J. Willems}, {\em Numerical upscaling for multiscale flow problems}, Ph.D Thesis, 2009.



\end{thebibliography}
\end{document}